\documentclass[11pt]{article}

\usepackage{framed}
\usepackage[normalem]{ulem}
\usepackage{amsmath,amsthm,amssymb,amsfonts}
\usepackage{tgtermes}
\usepackage[T1]{fontenc}
\usepackage{graphicx}
\usepackage{authblk}
\usepackage[margin = 1in]{geometry}
\usepackage{bm}           

\newtheorem{proposition}{Proposition}
\newtheorem{theorem}[proposition]{Theorem}
\newtheorem{corollary}[proposition]{Corollary}
\newtheorem{lemma}[proposition]{Lemma}

\newtheorem{assumption}[proposition]{Assumption}

\newtheorem{example}[proposition]{Example}

\newcommand{\Eb}{\mathbb{E}}

\newcommand{\ds}{\text{\sf d}}
\newcommand{\Ds}{\text{\sf D}}

\newcommand{\Ff}{\mathfrak{F}}
\newcommand{\Hc}{\mathcal{H}}

\newcommand{\Lc}{\mathcal{L}}

\newcommand{\Nb}{\mathbb{N}}

\newcommand{\Pc}{\mathcal{P}}
\newcommand{\Rb}{\mathbb{R}}
\newcommand{\Zb}{\mathbb{Z}}
\newcommand{\Uc}{\mathcal{U}}
\newcommand{\Xc}{\mathcal{X}}
\newcommand{\Fft}{\tilde{\mathfrak{F}}}
\newcommand{\betat}{\tilde{\beta}}
\DeclareMathOperator{\intt}{int}
\DeclareMathOperator{\co}{conv}

\begin{document}

\title{Stability and Sample-based Approximations of Composite Stochastic Optimization Problems
\thanks{This work was partially supported by the Office of Naval Research under grant no.~N00014-21-1-2161.}
}
\author[*]{Darinka Dentcheva}
\author[*]{Yang Lin}
\author[$\dag$]{Spiridon Penev}
\affil[*]{Department of Mathematical Sciences, Stevens Institute of Technology, Hoboken, NJ 03070, USA}
\affil[$\dag$]{School of Mathematics and Statistics and UNSW Data Science Hub, UNSW Sydney, Sydney, 2052 NSW, Australia}

\renewcommand\Authands{ and }

\maketitle

\textbf{Abstract: }
Optimization under uncertainty and risk is indispensable in many practical situations. Our paper addresses stability of optimization problems using composite risk functionals which are subjected to measure perturbations. Our main focus is the asymptotic behavior of data-driven formulations with empirical or smoothing estimators such as kernels or wavelets applied to some or to all functions of the compositions.  We analyze the properties of the new estimators and we establish strong law of large numbers, consistency, and bias reduction potential under fairly general assumptions. Our results are germane to risk-averse optimization and to data science in general.

\textbf{Keywords: }
stochastic programming, bias, coherent measures of risk, kernel estimation, wavelet estimation, consistency, strong law of large numbers

\section{Introduction}

Optimization under uncertainty and risk is ubiquitous in practical situations.
Multitude of papers in the area of machine learning, business, engineering, and other areas  address the properties and the numerical approached to optimization under uncertainty and risk. Very frequently, the problem formulation uses observed or simulated data. Most of the existing literature deals with stochastic optimization problems of the following general structure:
\begin{equation}
\label{p:basic-rn}
    \min_{u \in U} \Eb \big[F(u,X)\big].
\end{equation}
We call problems of this type risk-neutral.
Here $X$ is a random vector defined on the probability space $(\Omega,\mathcal{F},\mathbb{P})$ with realizations in $\Xc\subseteq\mathbb{R}^m$ and with a finite $p$ moment, $p\geq 1$. We denote
the set of all $m$-dimensional random vectors defined on $(\Omega,\mathcal{F},\mathbb{P})$ with finite $p$ moments by $\Lc_p\big(\Omega,\mathcal{F},\mathbb{P};\Rb^m\big)$.
In \eqref{p:basic-rn}, $U$ is a nonempty closed subset of $\mathbb{R}^n,$ representing the feasible decisions. The objective function $F:\Rb^n\times\Rb^m\to\Rb$ is assumed to be sufficiently regular for the expectation to be well defined and finite valued for all $u\in U$.

While problems of form \eqref{p:basic-rn} are well investigated, our focus is  placed on objective functions given by composite functionals of the following form:
\begin{equation}
\label{f:rho}
\varrho[u,X] =  \mathbb{E}\left[f_1 \left(u,\mathbb{E}[f_2(u,\mathbb{E}[\ldots f_k(u,\mathbb{E}[f_{k+1}(u,X)],X \right) \right] \ldots,X)],X)],
\end{equation}
where $u$ is the decision vector and the random vector $X$ comprises the random data.
The vector functions $f_j:\Rb^n\times{\mathbb{R}}^{m_j} \times {\mathbb{R}}^m \rightarrow {\mathbb{R}}^{m_{j-1}}$, $j=1,\cdots,k$ with $m_0=1$ and $f_{k+1}:\Rb^n\times{\mathbb{R}}^m \rightarrow {\mathbb{R}}^{m_k}$ are assumed continuous with respect to the first argument. The probability measure induced by $X$ is denoted by $P$ and we assume throughout the paper that the functions $f_j$, $j=1,\dots, k+1$ are $P$- integrable with respect to their last argument.
The motivation for this structure comes from the fact that many coherent measures of risk may be cast in this form (c.f., \cite{dentcheva2017statistical}). Recall that coherent measures of risk are functionals $\varrho: \Lc_p\big((\Omega,\mathcal{F},\mathbb{P});\Rb\big)\rightarrow \Rb\cup\{+\infty\}\cup\{-\infty\}$, which are monotonic with respect to the $\mathbb{P}$-a.s. order, convex, positively homogeneous, and satisfy $\varrho (X+c) = \varrho (X) +c$ for all $X\in \Lc_p((\Omega,\mathcal{F},\mathbb{P});\Rb)$ and all constants $c$ (see, e.g., \cite{PflRom:07,mainbook}.) Of course, we can estimate statistically only law-invariant measures of risk but we still keep the commonly accepted notation of $\varrho[u,X]$ instead of $\varrho[u,P].$ Problems arising in machine learning deal with composite optimization as well (e.g., \cite{Duchi2018,machine2}.) The composite structure allows more general point of view and may be of interest in its own right.

Suppose a sample $X_1, X_2,\dots, X_N$ of the random vector $X$ is available.
The most popular approach to problem \eqref{p:basic-rn} is the sample average approximation (SAA), which suggests to solve  the empirical counterpart of \eqref{p:basic-rn} by minimizing $\frac{1}{N}\sum_{i=1}^N F(u,X_{i})$. It is well-known that the optimal value of the SAA problem
suffers from a downward bias. Statistical inference for sample-based problems using expectations and other statistical risk functionals which are linear in probability or of classical character (e.g., moment estimation) are thoroughly investigated in the literature.

Our goals are to extend the theory of stability of stochastic optimization problems with respect to measure perturbation to the case of composite functions, as well as to address some of the questions arising in sample-based composite functional optimization. First, we establish two results about qualitative stability with respect to general measure perturbation for optimization problems with objectives of form \eqref{f:rho}. Further, we analyze the plug-in estimators and smoothed estimators with particular attention to kernel-based estimators. We identify conditions which allow us to establish consistency and strong law of large numbers for the optimal value and the optimal solutions of the sample-based problems.  We also discuss the possibility to smooth only some parts of the composite functional. Finally, we analyze the bias of the empirical and the smoothed estimators in composite stochastic optimization problems. Special attention is paid to the risk-averse optimization problems, in which the composite functionals represent higher order coherent measures of risk.

While properties of smoothed estimators are widely investigated, their application to composite functionals in sample-based optimization problems brings new issues to the fore. We refer to the following work, which was essential to the developments of our analysis:  \cite{tucl,E-consistency,wied2012consistency,Silverman1978weak,einmahl2005uniform,gine2004weighted}.
The empirical version of an optimization problem with composite objective is analyzed in \cite{dentcheva2017statistical}, where central limit theorems have been established.
Another study addressing compositions of similar type is presented in \cite{ermoliev2013sample}, where the asymptotics of certain specific composite problems are investigated. Related work is presented also in \cite{guigues2012sampling}.
The study \cite{Pang2019} focuses on non-convex problems of composite nature which arise in machine learning and analyzes their consistency.
We refer to \cite{romisch2003stability} and to \cite{pflug2003} for a comprehensive review on the asymptotic behavior of stochastic optimization problems; see also \cite{shapir:00,norkin1992convergence,pflug1998glivenko,dentcheva2013stability}.
Variance and bias reduction are discussed in the context of two-stage problem in \cite{sen2016mitigating}. Kernel estimators applied to the data in the context of stochastic programming are analyzed in \cite{growe1992stochastic,dentcheva2021bias}.
A detailed analysis of stochastic average approximation models and associate statistical inference for sample-based optimization is contained in \cite[Chapter 5]{mainbook}.  Statistical estimation of some measures of risk is discussed in \cite{Belomestny2012,Jones2003,Brazauskas2008,dentcheva2010,dentcheva2011mean,shapiro2013,mainbook,rachev2002quantitative,pflug2010asymptotic}.

In our earlier work \cite{dentcheva2021bias}, we have proposed smooth estimators for the sampled data in optimization problem of form \eqref{p:basic-rn}.  We have shown that the new formulations provide less biased estimation of the optimal value under certain assumptions while the standard error of the estimator remains controlled.  These assumptions are satisfied for many problems arising in statistics and stochastic optimization such as regression models of various types, classification problems, portfolio optimization using average value-st-risk and others.
Our experience with modern risk management problems shows that many of those problems exhibit downward bias, which diminishes very slowly, in fact, slower than in problems of form \eqref{p:basic-rn}.  For composite functionals, the bias could still be significant at large sample sizes. As underestimating the riskiness in risk management might lead to substantial losses, it is of practical interest to reduce this bias. Realistic risk-averse decision problems usually depend on high-dimensional data, as well as high-dimensional decision vector $u$, and therefore, the minimization itself is computationally very demanding. Hence applying bias-reduction methods such as jackknife or bootstrap (\cite{EfronAnnals}) may be numerically very expensive.

Our paper is organized as follows. Section \ref{s:stability} contains stability results for the optimal value and the optimal solutions of problems with objectives of form \eqref{f:rho}, when the probability measure $P$ is subjected to perturbations.
Section \ref{s:empiricalestimators} discusses statistical estimators of composite functionals and of the sample-based optimization problems. Consistency analysis is discussed for the empirical and the smooth estimators, which include kernel-based and wavelet-based estimators.  Section~\ref{s:bias} contains analysis of the bias and provides a comparison of the bias
in empirical sample-based optimization and the bias in optimal value of the kernel-based and wavelet-based composite optimization problems.  Section \ref{s:application}  discussed the application of our results to coherent measures of risk in portfolio optimization. Numerical experiments are reported in Section \ref{s:numerical}.


\section{Stability of optimization problems with composite functionals with respect to measure perturbations}
\label{s:stability}

The main object of our study are the optimization problem of the following form:
\begin{equation}
\label{p:composite}
    \vartheta = \min_{u \in U}\;  \varrho[u, X],
\end{equation}
where $\varrho[u, X]$ is defined in \eqref{f:rho}.
The functions $f_j$, $j=1,\cdots,k$ are assumed continuous with respect to the first two arguments and $f_{k+1}$ is assumed continuous with respect to the first argument.
The set of optimal solutions in \eqref{p:composite} is denoted by $S$, e.g.,
\[
S  = \{ u\in U: \varrho[u, X]= \vartheta \}.
\]
We assume throughout the paper that the set $S$ is non-empty and bounded.
\begin{example}
{\rm
Let the random returns of $m$ securities be gathered in a random vector $X$. Our portfolio is given by a vector $u$ representing the allocation of the available capital $K$. The set $U$ stands for the restrictions
on our potential allocations, e.g.,
\[
U=\{ u\in \Rb^m : \sum_{i=1}^m u_i =K,\; l_i\leq u_i \leq  b_i\},
\]
where $l_i$ and $u_i$ are lower and upper bounds, respectively, for the investment in the $i$-th security.
We optimize a combination of the mean return with its mean-semi-deviation of order $p\ge 1$ or with a higher order inverse measures of risk in order to determine our portfolio.
For a random variable $Y$, representing losses the mean-semi-deviation of order $p\ge 1$ has the form
\[
\varrho_1[Y] = \Eb[Y] + \kappa \Big(\Eb\big[\big(\max\{0,\Eb[Y]-Y\}\big)^p\big]\Big)^{\frac{1}{p}},
\]
where $\kappa \in [0,1]$.
We define $f_i:\Rb^m\times\Rb\times\Rb^m\to\Rb$, $i=1,2,$ and $f_3:\Rb^m\times\Rb^m\to\Rb$ as follows:
\begin{align*}
f_1(u,\eta_1,x) = -\langle u, x\rangle + \kappa\eta_1^{\frac{1}{p}},\quad
f_2(u,\eta_2,x) = \big(\max\{0, \eta_2 -\langle u, x\rangle\}\big)^p,\quad 
f_3(u,x) = -\langle u, x\rangle. 
\end{align*}
Then the portfolio optimization problem has the form
\begin{equation}
\label{port-msd}
\min_{u\in U} \Eb\Big[f_1\Big(u,\Eb\big[f_2\big(u,\Eb[f_3(u,X)], X\big)\big], X\Big)\Big]
\end{equation}
This problem reduces to \eqref{p:basic-rn} for $p=1$ but cannot be represented as an expected value optimization when $p>1$. We note that the problem has a unique solution for $p>1$, and, hence, the assumption about $S$ is satisfied.

Another choice of risk control in portfolio optimization is the use of inverse measures of risk. Those measures have the following structure:
\[
\varrho[Y]=\min_{z \in \mathbb{R}}\left\{ z+\frac{1}{\alpha} [\mathbb{E}(\max(0,Y-z)^q)]^{1/q}\right\}.
\]
In a portfolio optimization problem, we minimize a convex combination of
the higher order risk measure and the negative of the expected portfolio return.
The functions $f_1:\Rb^{m+1}\times\Rb^2\times\Rb^m$ and
$f_2:\Rb^m\times\Rb^m\to\Rb^2$ are defined as follows:
\begin{align*}
&f_1(u,\eta,x)=(1-\kappa)\eta_1+ \kappa \big(u_0+\frac{1}{\alpha}\eta_2^{1/q}\big),\\
&f_2(u,x)=\begin{pmatrix}
-\langle x,u\rangle\\
[\max(0,-\langle x,u\rangle-u_0)]^q
\end{pmatrix},
\end{align*}
where $\eta=(\eta_1,\eta_2).$
Here $\kappa\in (0,1)$ provides the weight of the risk measure in the objective.
The new optimization problem has the form
\begin{equation}
\label{port-hor}
\min_{u\in U} \Eb\big[f_1\big(u, \Eb[f_2(u,z, X )], X\big)\big].
\end{equation}
Problem \eqref{port-hor} also reduces to \eqref{p:basic-rn} for $q=1$ but cannot be represented as an expected value optimization when $q>1$. It also has a unique solution for $q>1.$
$\blacksquare$
}
\end{example}

We shall study stability
of the composite objective, the optimal value, and the optimal solution of problem \eqref{p:composite} when the measure $P$ is subjected to perturbations which may be different at the different levels of nesting. The notation $\Pc(\Xc)$ stands for the set of probability measures on $\Xc.$

For two sets, $A, B\subset\Rb^n$, the one-sided distance of $A$ to $B$ is defined as follows:
\[
\ds(A,B) = \sup_{x\in A} d(x,B) = \sup_{x\in A} \;\inf_{y\in B} \|x-y\|.
\]
The Pompeiu-Hausdorff distance between the sets is defined as
\[
\Ds(A,B) = \max \big\{ \ds (A,B), \ds (B,A)\big\}
\]
The following functions and sets will play a role in our discussion. For a measure $Q\in\Pc(\Xc)$, we define
\begin{equation}
\label{f:barfj-u-mu}
\begin{aligned}
     \bar{f}_j^Q(u,\eta_j)& =\int_{\mathcal{X}}f_j(u,\eta_j,x)\,{Q}(dx), \ \ \ \ j=1,\cdots,k \\
      \bar{f}_{k+1}^Q(u) & =\int_{\mathcal{X}}f_{k+1}(u,x)\,{Q}(dx)
\end{aligned}
\end{equation}
We fix a sufficiently large compact set $\Uc$ such that $S\subset \Uc\subset U.$
Further, we fix compact sets $I_1\subset\Rb^{m_1},\cdots,I_k\subset\Rb^{m_k}$ such that $\bar{f}_{j+1}(\Uc,I_{j+1}) \subset \intt(I_j)$, $j=1,\cdots,k-1$, and $\bar{f}_{k+1}(\Uc) \subset \intt(I_k)$, where $\intt(I_j)$ stands for the interior of $I_j.$ Without loss of generality, we assume that $\Uc$ and $I_j$, $j=1,\dots $ are convex sets.
We define the space:
\begin{equation*}
    \mathcal{H}={\mathcal{C}}_1(\Uc \times I_1) \times {\mathcal{C}}_{m_1}(\Uc \times I_2) \times \cdots \times {\mathcal{C}}_{m_{k-1}}(\Uc \times I_k) \times {\mathcal{C}}_{m_k}(\Uc)
\end{equation*}
where  ${\mathcal{C}}_{m_{j-1}}$ is the space of ${\mathbb{R}}^{m_{j-1}}$-valued continuous function on $\Uc \times I_j$, equipped with the supremum norm. The space $\Hc$ is equipped with the product norm.
We define $I = I_1\times I_2\times\cdots I_k$ and $d=m_0+m_1+\dots+m_k$ (recall $m_0=1$). For all $u\in\Uc$ and for all
$\eta= (\eta_1,\dots, \eta_k)\in I$ with $\eta_j\in I_j$, $j=1,2,\dots,k$, we define
\begin{gather*}
\bar{\mathbf{f}}^{(Q^1..Q^{k+1})}(u,\eta)= (\bar{f}^{Q^1}_1(u,\eta_1),\bar{f}^{Q^2}_2(u,\eta_2),\cdots,\bar{f}_k^{Q^k}(u,\eta_k),\bar{f}_{k+1}^{Q^{k+1}}(u))^{\top}\\
\mathbf{f}(u,\eta,x)= (f_1(u,\eta_1,x),f_2(u,\eta_2,x),\cdots,f_k(u,\eta_k,x),f_{k+1}(u,x))^{\top}\
\end{gather*}
If $Q^j=Q$ for all $j=1,\dots k+1$, we write $\bar{\mathbf{f}}^Q.$\\
Denoting the closed convex hull of $\Xc$ by $\co(\Xc)$,
we define the following set of functions.
\[
 \Ff_0 = \Big\{ f_j(u,\eta_j,\cdot):\co(\Xc)\rightarrow \Rb^{m_{j-1}}, f_{k+1}(u,\cdot): \co(\Xc)\rightarrow \Rb^{m_k}, u\in \Uc, \eta_j\in I_j, j=1,\dots, k \} \\
\]
The set $\Pc(\Xc)$ is equipped with the metric $\beta_0(Q,\tilde{Q})$, defined as follows:
\begin{equation}
\label{beta-distance}
	\beta_0(Q,\tilde{Q}) = \sup_{g\in\Ff_0} \, \Big| \int_\Xc g(x) dQ(x) - \int_\Xc g(x) d\tilde{Q}(x)\Big|
\end{equation}

Additionally, we introduce two other sets of functions.
The set of all Lipschitz-continuous and bounded functions on $\co(\Xc):$
\[
\Ff = \{g:\co(\Xc)\rightarrow \Rb: \big|g(x)-g(x')\big|\leq \|x-x'\|,\;\; \forall x,x'\in\co(\Xc),\;\;
\sup_{x\in\co(\Xc)} \big|g(x)\big|\leq 1\}.
\]
Here $\|x\|$ is the Euclidean norm in $\Rb^m.$
The respective metric $\beta(Q,\tilde{Q})$ on $\Pc(\Xc)$ is defined as follows:
\[
\beta(Q,\tilde{Q}) = \sup_{g\in\Ff} \, \Big| \int_\Xc g(x) dQ(x) - \int_\Xc g(x) d\tilde{Q}(x)\Big|
\]
It is well-known that  $\beta(Q,\tilde{Q})$  metrizes the weak convergence on $\Pc(\Xc)$.

Given functions $w_i:\Rb_+\to \Rb_+$, $i= 1,\dots, d$  such that $\lim_{t\downarrow 0} w_i(t) = w_i(0)=0$, we introduce the class of functions $\Fft$ and the respective metric $\betat$ as follows:
\begin{align*}
 \Fft = \Big\{ &\mathbf{f}_{i}(u,\eta,\cdot):\co(\Xc)\rightarrow \Rb,\quad u\in \Uc,\; \eta\in I,\; i=1,\dots, d,: \\
 &\big|\mathbf{f}_{i}(u,\eta,x)-\mathbf{f}_{i}(u,\eta,x')\big|\leq w_i\big( \|x-x'\|\big),\;\; \forall x,x'\in\co(\Xc),\; i= 1,\dots,d \Big \} \\
\betat(Q,\tilde{Q}) &= \sup_{g\in\Fft} \, \Big| \int_\Xc g(x) Q(dx) - \int_\Xc g(x) \tilde{Q}(dx)\Big|
\end{align*}
We note that $\Fft$ consists of real-valued functions that admit a given modulus of continuity.
Observe that $\Ff\subset\Fft$ entails that every sequence of measures converging with respect to $\betat$ also converges with respect to $\beta$.

The set of all natural numbers is denoted by $\Nb$. For a sequence of measures $\{Q_N^j\}$, $j=1,\dots k+1$, $N\in\Nb$, the approximate (measure-perturbed) problems are defined as follows:
 \begin{equation}
 \label{p:approximate}
 \begin{aligned}
\vartheta^{(Q_N^1..Q_N^{k+1})} &= \min_{u\in U} \varrho^{(Q_N^1..Q_N^{k+1})}[u,X] \quad{where}\\
\varrho^{(Q_N^1..Q_N^{k+1})}[u,X] &= \bar{f}_1^{Q_N^1}\Big(u, \bar{f}^{Q_N^2}_2\big(u,\cdots \bar{f}^{Q_N^k}_k(u, \bar{f}^{Q_N^{k+1}}_{k+1}(u))\cdots\big)\Big),\\
S^{(Q_N^1..Q_N^{k+1})} &= \big\{ u\in U:\; \varrho^{(Q_N^1..Q_N^{k+1})} [u,X] = \vartheta^{(Q_N^1..Q_N^{k+1})} \big\}.
\end{aligned}
\end{equation}
\begin{assumption}
Given a sequence of measures $\{Q_N^j\}$, it holds $S^{(Q_N^1..Q_N^{k+1})}\subset \Uc$ for $N$ large enough.
\end{assumption}

\begin{theorem}
\label{metric-main}
Assume that
the sequences of measures $Q_N^j$, $j=1,\dots k+1$ are such that $\lim _{N\to\infty} \beta_0(Q_N^j, P)=0$ for all $j=1,\dots, k+1$ and let Assumption 1 be satisfied for this approximation sequence.
Then\\
$\varrho^{(Q_N^1..Q_N^{k+1})}[u,X]\xrightarrow[N\to\infty] {}\varrho [u,X]$ for every $u\in \Uc$, $\vartheta^{(Q_N^1..Q_N^{k+1})}\xrightarrow[N\to\infty]{} \vartheta$,
 and  $\ds(S^{(Q_N^1..Q_N^{k+1})},S) \xrightarrow[N\to\infty]{} 0$.
Additionally, if problem \eqref{p:composite} has a unique solution $\hat{u}$, then the Pompeiu–Hausdorff distance $\Ds\Big(S^{(Q_N^1..Q_N^{k+1})},S\Big)$ converges to zero.
\end{theorem}

\begin{proof}
We define the function $h: \Uc\times I\to \Rb^d$ as follows.
\begin{equation}
\label{directionh}
h^{(Q_N^1..Q_N^{k+1})}(u,\eta)= \big(\bar{f}_1^{Q_N^1}(u,\eta_1),\bar{f}^{Q_N^2}_2(u,\eta_2),\dots, \bar{f}^{Q_N^k}_k(u,\eta_k), \bar{f}^{Q_N^{k+1}}_{k+1}(u)\big)^{\top}
\end{equation}
We have $\lim_{N\to\infty}\sup_{1\leq j\leq k+1} \beta(Q_N^j,P)=0, $
which implies the uniform (w.r.to $(u,\eta)\in \Uc\times I$) convergence of $h^{(Q_N^1..Q_N^{k+1})}(u,\eta)\xrightarrow[N\to\infty]{}
\bar{\mathbf{f}}^P(u,\eta)$, that is, the convergence of $h^{(Q_N^1..Q_N^{k+1})}(u,\eta)$ in the space $\Hc$.
The continuity of the functions $f_j$, $j=1,\dots k$ with respect to their second argument implies the uniform (w.r.to $u\in \Uc$) convergence of
the composition
\begin{multline*}
          \varrho^{(Q_N^1..Q_N^{k+1})}[u,X]= \bar{f}_1^{Q_N^1}\Big(u,\bar{f}_2^{Q_N^2}\big(u,\cdots \bar{f}_k^{Q_N^k}(u,\bar{f}_{k+1}^{Q_N^{k+1}}(u))\cdots\big)\Big)
            \xrightarrow[N\to\infty]{}\\
            \bar{f}_1^P\Big(u,\bar{f}_2^P\big(u,\cdots \bar{f}_k^P(u,\bar{f}_{k+1}^P(u))\cdots\big)\Big) = \varrho[u,X].
        \end{multline*}
This shows the statement about the convergence of the composite risk functional for any fixed argument.
Due to the assumptions, problems \eqref{p:composite} and \eqref{p:approximate} have a non-empty solution sets for $N$ sufficiently large and moreover the sets $S^{(Q_N^1..Q_N^{k+1})}$ are non-empty.

We define the functional $\Psi:\Uc \times \Hc \to\Rb$ by setting
\[
\Psi(u,h) = h_1\Big(u,h_2\big(u,\cdots h_k(u,h_{k+1}(u))\cdots\big)\Big)
\]
For a given parameter $h\in\Hc$, we consider the optimization problem
\begin{equation}
\label{p:abstract}
\min_{u \in \Uc} {\Psi}(u, h)
\end{equation}
and let $S(h)$ stand for the set of optimal solutions of \eqref{p:abstract}. Under our assumptions, we have
\[
\vartheta= \min_{u \in \Uc} {\Psi}(u, \bar{\mathbf{f}}^P(u,\eta))\;\;\text{ and }\;\;
\vartheta^{(Q_N^1..Q_N^{k+1})}=\min_{u \in \Uc} {\Psi}(u, h^{(Q_N^1..Q_N^{k+1})}(u,\eta)).
\]
We apply \cite[Theorem 4.2.2]{bank1982non} for the parameter value $h=\bar{\mathbf{f}}^P$.
Observe that $\Psi(\cdot, \cdot)$ is continuous by the definition of $\Psi$ and the compactness of $\Uc.$ Then the first two statements of \cite[Theorem 4.2.2]{bank1982non}
imply that the mapping $h\to \min_{u \in U} \Psi(u, h)$ is continuous.
Since $h^{(Q_N^1..Q_N^{k+1})} \xrightarrow[N\to\infty]{} \bar{\mathbf{f}}^P$, we conclude that $\vartheta^{(Q_N^1..Q_N^{k+1})} \xrightarrow[N\to\infty]{} \vartheta$.
Furthermore, statement (3) of \cite[Theorem 4.2.2]{bank1982non} holds as well, implying that the set-valued mapping
$h\rightrightarrows S(h)$ is upper-semicontinuous at $\bar{\mathbf{f}}^P$. This means that $\ds\big(S(h^{(Q_N^1..Q_N^{k+1})}),S(\bar{\mathbf{f}}^P)\big) \xrightarrow[N\to\infty]{} 0$ when $h^{(Q_N^1..Q_N^{k+1})}\xrightarrow[N\to\infty]{}\bar{\mathbf{f}}^P.$ Furthermore, when $S(\bar{\mathbf{f}}^P)$ contains only one element, $\hat{u},$ then
\[
\Ds\big(S(h^{(Q_N^1..Q_N^{k+1})}),S(\bar{\mathbf{f}}^P)\big)=
\sup_{u\in S(h^{(Q_N^1..Q_N^{k+1})})} \|u-\hat{u}\|
= \ds(S^{(Q_N^1..Q_N^{k+1})},S).
\]
Thus, we infer the last claim of the theorem.
 
\end{proof}

 Recall that the convergence $\ds\big(S(h^{(Q_N^1..Q_N^{k+1})}),S(\bar{\mathbf{f}}^P)\big) \xrightarrow[N\to\infty]{} 0$ implies $\lim\sup_{N\to\infty} S(h^{(Q_N^1..Q_N^{k+1})})\subseteq S,$ that is, all accumulation points of sequences $u^N\in S(h^{(Q_N^1..Q_N^{k+1})})$ belong to $S$. Note that all such sequences have accumulation points due to the boundedness of $S.$

\begin{theorem}
\label{metric-Lipschitz}
Assume that the sequence of measures $Q_N^j$, $j=1,\dots k+1$ are such that $Q_N^j\to P$ weakly as $N\to \infty$ and Assumption 1 holds. Suppose one of the following conditions:
\begin{itemize}
	\item[(a)] the functions $f_j(u,\eta_j,\cdot)$, $j=1,\cdots,k$ and $f_{k+1} (u,\cdot)$ belong to $\Ff$ for all $(u,\eta)\in U\times I$;
	\item[(b)] the functions $f_j(u,\eta_j,\cdot)$, $j=1,\cdots,k$ and $f_{k+1} (u,\cdot)$ belong to $\Fft$ and, additionally,
the sequences of measures $Q_N^j$, $j=1,\dots k+1$ satisfy
$\lim_{N\to\infty}\int_\Xc \|x\| \; Q_N^j(dx) = \int_\Xc \|x\| \; P(dx)$ with all integrals being finite.
\end{itemize}
Then the conclusions of Theorem~\ref{metric-main} hold.
\end{theorem}

\begin{proof}
We consider the case (a) first. \\
Since $Q_N^j\to P$ weakly, we have $\lim_{N\to\infty}\sup_{1\leq j\leq k+1} \beta(Q_N^j,P)=0. $ Since $\Ff_0\subset\Ff,$
this implies $\lim_{N\to\infty}\sup_{1\leq j\leq k+1} \beta_0(Q_N^j,P)=0. $
The claim follows by Theorem \ref{metric-main}.

Now, we turn to the case of condition (b).

It is well known that if a real-valued function, which is defined on a convex subset of a metric space, admits a modulus of continuity $w$, then $w$ can be selected to be subadditive. Therefore, we shall assume without loss of generality that all functions $w_\ell$, $\ell=1,\dots, d$ are subadditive. We define $\bar{w}(t) = \max_{1\leq \ell \leq d} w_\ell(t)$ for all $t\geq 0,$ which is a common modulus of continuity for all functions in $\Fft.$ The function $\bar{w}(\cdot)$ is subadditive as well. Indeed, for any $0\leq s < t,$ we have
\begin{multline*}
\bar{w}(s+t) = \max_{1\leq \ell \leq d} w_\ell(s+t) \leq \max_{1\leq \ell \leq d} \big( w_\ell(s) + w_\ell(t)\big)
\leq \max_{1\leq \ell \leq d} w_\ell(s) + \max_{1\leq \ell \leq d} w_\ell(t) = \bar{w}(s) +\bar{w}(t).
\end{multline*}
Any subadditive modulus of continuity is continuous and
has a sublinear growth, i.e., positive constants $\alpha$ and $\beta$ exist such that $\bar{w}(t) \leq \alpha t+\beta$ for all $t\geq 0$.
Indeed, continuity follows from subadditivity and the continuity of $w$ at 0. For any positive integer $s$, we have $\bar{w}(s)\leq \bar{w}(s-1) +  \bar{w}(1)\leq \cdots s \bar{w}(1)$. Then any $r\geq 0$ can be represented as a sum of an integer $s$ and a number $\alpha\in [0,1)$.  We obtain
\[
\bar{w}(r)\leq \bar{w}(s) + \bar{w}(\alpha) \leq s \bar{w}(1) + \bar{w}(\alpha)\leq r\bar{w}(1) + \sup_{0\leq t\leq 1} \bar{w}(t),
\]
which shows the claim with $\alpha=\bar{w}(1)\geq 0$  and $\beta= \sup_{0\leq t\leq 1} \bar{w}(t). $
This implies that for all pairs $(j,i)$ with $j=1,\dots k+1, \, i=1,\dots m_{j-1}$ the following relation holds
\begin{multline*}
\int_\Xc |f_{j,i}(u,\eta_j, x)| Q_j(dx) \leq \int_\Xc |f_{j,i}(u,\eta_j, x) - f_{j,i}(u,\eta_j, \bar{x})| Q_j(dx) + f_{j,i}(u,\eta_j, \bar{x})  \\
\leq \int_\Xc \bar{w}(\|x-\bar{x}\|) Q_j(dx) + f_{j,i}(u,\eta_j, \bar{x}) \leq  \int_\Xc \big(\alpha\|x- \bar{x}\| +\beta\big)\;  Q_j(dx) + f_{j,i}(u,\eta_j, \bar{x})\\
 \leq  \alpha\int_\Xc\|x\|  Q_j(dx) + \alpha\|\bar{x}\| + \beta + f_{j,i}(u,\eta_j, \bar{x})<\infty.
\end{multline*}
Here $\bar{x}\in\Xc$ is an arbitrary fixed point.
Analogously, for all $i=1,\dots m_k$
\[
\int_\Xc |f_{k+1}(u,x)| Q_{k+1}(dx)\leq \alpha\int_\Xc \|x\| Q_{k+1}(dx) + \alpha\|\bar{x}\| +\beta + f_{k+1,i}(u,\eta_j, \bar{x}) \leq \infty.
\]
The argument also shows that all functions in $\Fft$ as well as $\bar{w}(\cdot)$ are bounded by the following function:
\[
G(x) = \alpha\|x\| + b,
\]
where $b= \alpha\|\bar{x}\| + \beta + \sup_{u\in\Uc, \eta\in I, i=1,\dots d} \mathbf{f}_{i}(u,\eta,\bar{x})$.
Since $Q^j_N$ converge to $P$ for $j=1,\dots k+1$, the set of measures $\{ P, Q^j_N, N\in\Nb, j=1,\dots k+1\}$ are uniformly tight by \cite[Proposition 9.3.4]{dudley2002real}. Therefore,  for every $\delta >0$, let  $\mathcal{K}_\delta\subset \Rb^m$ be a compact set such that
$P(\mathcal{K}_\delta)\geq 1-\delta$ and $Q^j_N(\mathcal{K}_\delta)\geq 1-\delta$ for all $N\in\Nb$.
This means that
\[
\lim_{r\to\infty}\sup_{j=1,\dots, k+1,\; N\in\Nb} \int_{\|x\|>r } b \; Q_N^j(dx) = 0.
\]
Now, we obtain the following:
\begin{align*}
\lim_{r\to\infty} &\,\sup_{j=1,\dots, k+1,\; N\in\Nb}\sup_{g\in\Fft} \int_{\|x\|>r } | g(x)| \; Q_N^j(dx)\\
&\,\leq \lim_{r\to\infty} \sup_{j=1,\dots, k+1,\; N\in\Nb} \int_{\|x\|>r } \alpha \|x\|  \; Q_N^j(dx) + \lim_{r\to\infty}\sup_{j=1,\dots, k+1,\; N\in\Nb}\int_{\|x\|>r }b \; Q_N^j(dx)\\
 &\,= \lim_{r\to\infty} \sup_{j=1,\dots, k+1,\; N\in\Nb} \int_{\|x\|>r } \alpha\|x\|  \; Q_N^j(dx)
\end{align*}
Since $\lim_{N\to\infty}\int_\Xc \|x\| \; Q_N^j(dx) = \int_\Xc \|x\| \; P(dx)$ by assumption and $\lim_{N\to\infty}\int_{\|x\|\leq r } \|x\|  \; Q_N^j(dx)= \int_{\|x\|\leq r } \|x\|  \; P(dx)$ by the weak convergence of measures, we obtain 
\[
\lim_{N\to\infty}\int_{\|x\|>r } \|x\| \; Q_N^j(dx) = \int_{\|x\|>r } \|x\| \; P(dx)
\] for all $r$ such that $P(\|x\| = r) =0. $ Since $\lim_{r\to\infty} \int_{\|x\|>r } \alpha\|x\| \; P(dx) =0,$ we infer that $N_0\in\Nb$ exists such that
$\lim_{r\to\infty} \sup_{j=1,\dots, k+1,\; N>N_0} \int_{\|x\|>r } \alpha \|x\|  \; Q_N^j(dx)= 0.$ This implies
\[
\lim_{r\to\infty} \sup_{j=1,\dots, k+1,\; N\in \Nb} \int_{\|x\|>r } \alpha \|x\|  \; Q_N^j(dx)= 0,
\] which
entails that $\Fft$ is a $P$-uniformity class and $\lim_{N\to\infty}\sup_{1\leq j\leq k+1} \betat(Q_N^j,P)=0. $
Now, the claim of the theorem follows by the same line of arguments as in Theorem \ref{metric-main}.
  \end{proof}


\section{Estimation of Composite Functionals}
\label{s:empiricalestimators}

We shall consider several statistical estimators of the
composite risk functional, the optimal value, and the optimal solution of problem \eqref{p:composite}.
Given independent and identically distributed $X_1, X_2,...$ realizations of $X$, ${\mathbf{X}}=\{X_i\}_{i=1}^{\infty}$, the convergence for almost all is understood with respect to the product probability ${P}^{\infty}={P}\times {P}\dots .$

\subsection{Empirical estimators}
The empirical estimator of the composite risk functional is the following
\begin{multline}
\label{f:rhoN}
\sum_{i_0=1}^N\frac{1}{N}\Big[f_1\Big(u,\sum_{i_1=1}^N\frac{1}{N}\big[f_2\big(u,\sum_{i_2=1}^N\frac{1}{N}[\cdots f_k(u,\sum_{i_k=1}^N\frac{1}{N}f_{k+1}(u,X_{i_k}),X_{i_{k-1}})]
\cdots,X_{i_1}\big)\big],X_{i_0}\Big)\Big]
  \end{multline}
In our setting, it is not justified to speak about sample average approximation because the objective is not representable as the expected value of a single function. That is why we call problem \eqref{p:approximate} for $Q_N^j= P_N$ for all $j=1,\dots k+1$ \emph{empirical composite optimization} problem.
We emphasize that we use \emph{the entire sample} for the estimation of each expected value at every level.

We show the consistency of the empirical estimators under weaker assumptions than those in the previous section. We use $\varrho^{(N)}_E [u,X]$ for the empirical estimator in formula \eqref{f:rhoN} and $\vartheta_E^{(N)}$,  and  $S_E^{(N)}$ for the corresponding optimal value and optimal solutions in problem \eqref{p:approximate} when $Q_N^j= P_N$.
\begin{proposition}
\label{empirical_SLLN}
Suppose Assumption 1 holds for $Q_N^j= P_N$ for all $j=1,\dots k+1$.
Assume $f_j(u,\eta_j, \cdot)$, $j=1,\dots, k$, and $f_{k+1}(u,\cdot)$ are uniformly bounded for all $u\in\Uc$ and for all $\eta_j \in I_j$ by a $P$-integrable function $g_j:\Rb^m\to \Rb$, i.e.,
$\|f_j(u,\eta_j, x)\|  \leq g_j(x)$ and $\|f_{k+1}(u, x)\|  \leq g_{k+1}(x)\quad\text{for all } x\in\Xc. $
Then
$\varrho_E^{(N)} [u,X]\xrightarrow[N\to\infty]{a.s.} \varrho [u,X]$ for every $u\in \Uc$,
$\vartheta_E^{(N)}\xrightarrow[N\to\infty]{a.s.} \vartheta$,
 and  $\ds(S_E^{(N)},S) \xrightarrow[N\to\infty]{a.s.} 0$.
Additionally, if problem \eqref{p:composite} has a unique solution $\hat{u}$, then $\Ds(S_E^{(N)},S) \xrightarrow[N\to\infty]{a.s.} 0$ as well.
\end{proposition}
\begin{proof}
We define a function $g:\Xc\to \Rb$
by setting
$g(x) = \max \{g_1(x),\dots,g_{k+1}(x)\}\quad x\in\Xc.$
It is an integrable function bounding uniformly the set of functions
$\Ff_0$, which entails that $\Ff_0$ is a Glivenko-Cantelli class. We define the perturbation $h^{P_N}$ by setting
$h^{P_N}_i = \frac{1}{N}\sum_{i=1}^{N} \bar{\mathbf{f}}_i(u,\eta,X_{i})$ for all $i=1,\dots, d$. We use $h^{P_N}$ as the perturbation defined in \eqref{directionh} and proceed with the same line of arguments as in the proof of Theorem~\ref{metric-main} with the additional invocation of the continuous mapping theorem.
  \end{proof}
Under these assumptions every solution of problem \eqref{p:approximate} is a strongly consistent estimator of the true solution.

\subsection{Smoothed estimators by convolutions}

We consider the smoothed estimators that are obtained by applying a convolution with a measure $\mu_N$ to the empirical measure $P_N$ associated with the sample at hand.
In \cite{tucl}, the notion of proper approximate convolutional identity is introduced. This is
a sequence of measures $\{\mu_N\}_{N=1}^\infty$, independent of $P_N$, such that
$\mu_N$ converge weakly to the point mass $\delta(0)$ when $N\to\infty$ and
for every $a>0$,
$\lim_{N\to\infty} |\mu_N| (\Rb^m \setminus [-a,a]^m) = 0$, with $|\mu_N|$ denoting the total variation of $\mu_N.$

We augment these conditions by assuming the following.
\begin{assumption}
The sequence of measures $\{\mu_N\}$ are independent of $P_N$, normalized  ($\mu_N(\Rb^m)=1$), and satisfying
\begin{itemize}
\item $\{\mu_N\}$ converges weakly to the point mass $\delta(0)$ when $N\to\infty$;
\item $\int_{\Rb^m} \|z\| \,d\mu_N(z)$ is finite and
 $\lim\limits_{r\to\infty} \lim\limits_{N\to\infty}\int_{\Rb^m:\|z\|> r } \|z\| \,d\mu_N(z) =0.$
\end{itemize}
\end{assumption}
This assumption is satisfied, if for example all $\mu_N$ have bounded support, or have densities $d(z)$ with respect to the Lebesque measure, with tails satisfying $d(z)\leq 1/\|z\|^{1+\varepsilon}$ for some $\varepsilon>0.$
The smooth estimator for the expectation of a function $g:\Rb^m\to\Rb$ is defined as follows:
\begin{equation}
\label{smoothE}
[P_N *\mu_N] g(X) = \frac{1}{N}\sum_{i=1}^N \int_{\Rb^m} g(X_i+z) \, d\mu_N(z).
\end{equation}
A special case is given by a kernel estimator of form:
\[
\frac{1}{Nh_N^m}\sum_{i=1}^N \int_{\Rb^m} g(x) K\Big(\frac{x-X_i}{h_N}\Big)\, dx,
\]
where $K$ is a $m$-dimensional density function with respect to the Lebesgue measure and $h_N>0$ is a smoothing parameter such that $\lim_{N\to\infty} h_N=0$. We have
 $d\mu_N(x) = \frac{1}{h_N^m} K\Big(\frac{x}{h_N}\Big)\,dx$.

The estimators $\mu_N$ may take more general form than the kernel estimator just defined for illustration (cf. \cite{tsybakov2008introduction,gine2004weighted}).
We also do not need to apply the convolution to all levels of nesting but we always use \emph{all} observations at every level of nesting.

When using kernels, we shall assume the following properties.
\begin{itemize}
\item[(k1)] The kernel $K$ of order $s> 1$ is a density function with respect to the Lebesgue measure satisfying
$\int\limits_{\Rb^m} y_l^j K(y)dy=0$ for $l=1,\cdots,m$, $j=1,\dots, \lfloor s\rfloor$ with $\lfloor s\rfloor$ being the largest integer smaller than $s.$
\item[(k2)]  The $s$-th order moment
$m_s(K) = \int\limits_{\Rb^m}\|y\|^s K(y)dy$ is finite.
\end{itemize}
Under assumptions (k1)-(k2), all moments
$m_\alpha(K) = \int\limits_{\Rb^m}\|y\|^\alpha K(y)dy$ for all $\alpha\in (0,s]$ are finite.

In order to avoid cluttering the notation, we shall omit the area of the integration when it does not lead to ambiguity. We use $\vartheta_K^{(N)}$, $\varrho^{(N)}_K [u,X]$, and  $S_K^{(N)},$ respectively, when the smoothed estimators use the same kernel for smoothing all functions in the composition.
Similarly, we use $\vartheta_E^{(N)}$, $\varrho^{(N)}_E [u,X]$, and  $S_E^{(N)}$ when only sample averages are used.

We shall show that the strong law of large numbers holds for the smoothed estimators and for those of mixed nature under relatively mild assumptions. Assume that the functions $f_j,$ $j=1,\dots k$ defining the problem belong to the set $\Fft.$  The index set $J\subseteq\{1,2,\dots,k+1\}$ contains all indices of the functions in the composition, where smoothing is applied.
We use the notation $\varrho_\mu^{(N,J)}$ for the estimator, in which $Q_N^j = P_N *\mu_N$ for $j\in J$ and $Q_N^j = P_N$ for $j\in \{1,2,\dots,k+1\}\setminus J.$ The corresponding optimal value and optimal solutions are denoted by $\vartheta_\mu^{(N,J)}$ and $S_\mu^{(N,J)}$, respectively.


\begin{theorem}
\label{SLLN_smoothed}
Let an index set $J\subseteq\{1,2,\dots,k+1\},$ and a sequence of measures $\{\mu_N\}$ satisfying Assumption 2 be given. Assume the following conditions
\begin{itemize}
	\item for $j\in J$, the functions $f_{j,i}(u,\eta_j,\cdot)\in \Fft$, $i=1,\dots, m_{j-1}$ for all $(u,\eta_j)\in \Uc\times I_j$; if $k+1\in J$, then $f_{k+1,i}(u,\cdot)\in\Fft$, $i=1,\dots, m_k$ for all $u\in U$.
 	\item for $j\not\in J$, $f_j(u,\eta_j, \cdot)$ as well as $f_{k+1}(u,\cdot)$ for $k+1\not\in J$ are uniformly bounded for all $(u,\eta\in\Uc\times I$ by a $P$-integrable function $g:\Rb^m\to \Rb$;
\end{itemize}
Then
$\varrho_\mu^{(N,J)} [u,X]\xrightarrow[N\to\infty]{a.s.} \varrho [u,X]$ for every $u\in \Uc$,
$\vartheta_\mu^{(N,J)}\xrightarrow[N\to\infty]{a.s.} \vartheta$,
 and  $\ds(S_\mu^{(N,J)},S) \xrightarrow[N\to\infty]{a.s.} 0$.
Additionally, if problem \eqref{p:composite} has a unique solution, then $\Ds(S_\mu^{(N,J)},S) \xrightarrow[N\to\infty]{a.s.} 0$ as well.
\end{theorem}
\begin{proof}
It is sufficient to show that the empirical and the smoothed estimators converge uniformly to  $\mathbf{\bar{f}}^P(u,\eta)$ for all $(u,\eta)\in U\times I$ for all functions $f_j$, $j=1,\dots ,k+1.$ Then the statement follows in the same way as in Theorem~\ref{metric-main} using additionally the continuous mapping theorem.
The perturbation function $h_\mu^{(N,J)}(u,\eta)$ is defined as follows; its $j$-th component is given by
\begin{equation}
\label{e:hmu}
[h_\mu^{(N, J)}]_j(u,\eta)=
\begin{cases}
\frac{1}{N}\sum_{i=1}^{N}\int f_j(u,\eta,X_i+ z) d\mu_N(z) &\text{ if } j\in J, j=1,\dots, k,\\
\frac{1}{N}\sum_{i=1}^{N} f_j(u,\eta,X_{i}) &\text{ if } j\not\in J, j=1,\dots, k,\\
\frac{1}{N}\sum_{i=1}^{N}\int f_{k+1}(u,X_{i}+ z) d\mu_N(z) &\text{ if } k+1\in J,\\
\frac{1}{N}\sum_{i=1}^{N} f_{k+1}(u,X_{i}+ z) &\text{ if } k+1\not\in J.\\
\end{cases}
\end{equation}
We cannot apply directly the results form \cite{tucl} because we do not assume that the smoothed functions are $\mu_N$-essentially bounded.

The difference between $h_\mu^{(N,J)}(u,\eta)$ and the expected value $\mathbf{\bar{f}}^P(u,\eta)$, can be bounded by the maximal difference of the respective components of these vector functions.
The $j$-th component of the function difference $\Delta_j^{(N)}(u,\eta)=[h_\mu^{(N,J)}(u,\eta) - \mathbf{\bar{f}}^P(u,\eta)]_j$ is given by
\begin{equation*}
\Delta_j^{(N)}(u,\eta)=
\begin{cases}
\frac{1}{N}\sum_{i=1}^{N}\int f_j(u,\eta,X_i+ z) \, d\mu_N(z) - \mathbf{\bar{f}}_j^P(u,\eta_j) & \text{ if } j\in J, j=1,\dots, k,\\
\frac{1}{N}\sum_{i=1}^{N} f_j(u,\eta,X_i) - \mathbf{\bar{f}}_j^P(u,\eta_j) &\text{ if } j\not\in J, j=1,\dots, k,\\
\frac{1}{N}\sum_{i=1}^{N} \int f_{k+1}(u,X_{i}+ z) \, d\mu_N(z) - \mathbf{\bar{f}}_{k+1}^P(u) & \text{ if } k+1\in J,\\
\frac{1}{N}\sum_{i=1}^{N} f_{k+1}(u,X_{i}) \, d\mu_N(z) - \mathbf{\bar{f}}_{k+1}^P(u)  &\text{ if } k+1\not\in J.\\
\end{cases}
\end{equation*}
For $j\not\in J$, we have $\lim_{N\to\infty}\sup_{(u,\eta)\in \Uc\times I} \Delta_j^{(N)}= 0$ as in Proposition~\ref{empirical_SLLN}.
To show the convergence of $\sup_{(u,\eta)\in \Uc\times I} \Delta_j^{(N)}$  for all $j\in J$ we shall show that the assumptions of Theorem~\ref{metric-Lipschitz} are satisfied. We only need verify that $\int_\Xc \|x\| \; \mu_N(dx) \xrightarrow[N\to\infty]{}\int_\Xc \|x\| \; P(dx).$ The following relations hold:
\begin{align*}
\lim_{N\to\infty} \int \|z\|\, d\mu_N(z) &= \lim_{N\to\infty} \int_{\|z\|\leq r} \|z\|\, d\mu_N(z) +
\lim_{N\to\infty} \int_{\|z\|> r} \|z\|\, d\mu_N(z)\\
&  = \int_{\|z\|\leq r} \|z\|\, dP(z)  + \lim_{N\to\infty} \int_{\|z\|> r} \|z\|\, d\mu_N(z).
\end{align*}
Letting $r\to\infty$ and using Assumption 2, we obtain
\begin{align*}
\lim_{N\to\infty} \int \|z\|\, d\mu_N(z) = \int_{\|z\|\leq r} \|z\|\, dP(z).
\end{align*}
We infer the uniform convergence of $h_\mu^{(N,J)}(\cdot)$ to $\bar{f}(\cdot).$ The statement follows by the same line of arguments as in the proof of Theorem~\ref{metric-main} invoking additionally the continuous mapping theorem.
  \end{proof}

We observe that the assumptions of Theorem~\ref{SLLN_smoothed} are easier to verify using the kernel estimators.
\begin{corollary}
\label{SLLN_kernel}
	Assume that the kernel function $K$ satisfies (k1)--(k2) and $d\mu_N(x) = \frac{1}{h_N^m} K\Big(\frac{x}{h_N}\Big)\,dx$ with $\lim_{N\to\infty} h_N=0$. Suppose the functions $f_{j,i}(u,\eta_j,\cdot)$, for $j\in J$, $i=1,\dots, m_{j-1}$ belong to $\Fft$ for all $(u,\eta_j)\in \Uc\times I_j$ and
$f_{k+1,i}(u,\cdot)\in\Fft$, $i=1,\dots, m_k$ for all $u\in U$
Then
$\varrho_K^{(N)} [u,X]\xrightarrow[N\to\infty]{a.s.} \varrho [u,X]$ for every $u\in \Uc$,
$\vartheta_K^{(N)}\xrightarrow[N\to\infty]{a.s.} \vartheta$,
 and  $\ds(S_K^{(N)},S) \xrightarrow[N\to\infty]{a.s.} 0$.
Additionally, if problem \eqref{p:composite} has a unique solution, then $\Ds(S_K^{(N)},S) \xrightarrow[N\to\infty]{a.s.} 0$ as well.
\end{corollary}
\begin{proof}
We only need to show that Assumption 2 is satisfied for the sequence
 $d\mu_N(x) = \frac{1}{h_N^m} K\Big(\frac{x}{h_N}\Big)\,dx$.
Indeed, the weak convergence condition is shown in \cite{tucl}. Additionally,
\[
\lim_{N\to\infty} \frac{1}{h_N^m} \int_{\|z\|> b } \|z\| K\Big(\frac{z}{h_N}\Big)\, dx =
\lim_{N\to\infty} \int_{\|y\|> \frac{1}{h_N}b } \|h_Ny\| K(y)\, dy\\
\leq  \lim_{N\to\infty} h_N m_1(K) =0,
\]
which completes the proof.
  \end{proof}

From a practical perspective, we may not have uniform modulus of continuity for the case of unbounded function but we still need to ensure consistency of the estimators. We shall show that  $\varrho_\mu^{(N,J)}[u,X]$ converges in probability to $\varrho[u,X]$ (written $\varrho^{(N,J)}[u,X]\xrightarrow[N\to\infty]{p}\varrho[u,X]$) under mild conditions.
\begin{theorem}
\label{t:consistency}
Let an index set $J\subseteq\{1,2,\dots,k+1\},$ and a sequence of measures $\{\mu_N\}$ satisfying Assumption 2 be given.
Assume that for $j=1,\dots, k$, $f_j(u,\eta_j, \cdot)$ and $f_{k+1}(u,\cdot)$ are continuous and uniformly bounded for all $u\in\Uc$ and for all $\eta_j \in I_j$ by a $P$-integrable function $g_j:\Rb^m\to \Rb$.
Then
$\varrho_\mu^{(N,J)} [u,X]\xrightarrow[N\to\infty]{p} \varrho [u,X]$ for every $u\in \Uc$,
$\vartheta_\mu^{(N,J)}\xrightarrow[N\to\infty]{p} \vartheta$,
 and  $\ds(S_\mu^{(N,J)},S) \xrightarrow[N\to\infty]{p} 0$.
Additionally, if problem \eqref{p:composite} has a unique solution, then $\Ds(S_\mu^{(N,J)},S) \xrightarrow[N\to\infty]{p} 0$ as well.
\end{theorem}
\begin{proof}
 We consider the vector function
 $h_\mu^{(N,J)}(\eta)$ defined in equation \eqref{e:hmu}.
We shall show that the probability for the difference between $h_\mu^{(N,J)}(u,\eta)$ and $\bar{f}(u,\eta)$  to exceed a positive number $\varepsilon$ converges to zero when $N\to\infty.$
We have
\begin{align*}
P\big[ &\sup_{(u,\eta)\in \Uc\times I}  \big\|h_\mu^{(N,J)}(u,\eta)- \bar{f}(u,\eta)\big\| \geq \varepsilon \big] \\
& \leq
 P\big[\sup_{(u,\eta)\in \Uc\times I}\, \big\|h_E^{(N)}(u,\eta)- \bar{f}(u,\eta)\big\| \big] + \sup_{(u,\eta)\in \Uc\times I}\, \big\| h_\mu^{(N,J)}(u,\eta) - h_E^{(N)}(u,\eta)\big\|\geq \varepsilon\big]\\
 & \leq
 P\big[\max\big(\sup_{(u,\eta)\in \Uc\times I}\, \big\|h_E^{(N)}(u,\eta)- \bar{f}(u,\eta)\big\| \big],\sup_{(u,\eta)\in \Uc\times I}\, \big\| h_\mu^{(N,J)}(u,\eta) - h_E^{(N)}(u,\eta)\big\|\big)\geq \varepsilon/2\big]\\
 & \leq
 P\big[\sup_{(u,\eta)\in \Uc\times I}\, \big\|h_E^{(N)}(u,\eta)- \bar{f}(u,\eta)\big\| \big]\geq \varepsilon/2\big]+
 P\big[\sup_{(u,\eta)\in \Uc\times I}\, \big\| h_\mu^{(N,J)}(u,\eta) - h_E^{(N)}(u,\eta)\big\|\big)\geq \varepsilon/2\big].
\end{align*}
Due to Theorem~\ref{SLLN_smoothed} the first term at the right-hand side converges to zero whenever $N\to\infty$. We show the convergence of the second term.
We use the fact that the set of measures $\{P, P_N, N=1,2,\dots\}$ are uniformly tight, as well as the measures $\{\mu_N, N=1,2,\dots\}$ by virtue of \cite[Proposition 9.3.4]{dudley2002real}. For every $\delta >0$, let
$\mathcal{K}^\delta_X\subset \Rb^m$ be a compact subset such that
$P(\mathcal{K}^\delta_X)\geq 1-\delta$ and $P_N(\mathcal{K}^\delta_X)\geq 1-\delta$ for all $N\in \Nb$ and let $\mathcal{K}^\delta_\mu\subset \Rb^m$ be a compact subset such that
$\mu_N(\mathcal{K}^\delta_\mu)\geq 1-\delta$ for all $N\in \Nb$. Without loss of generality, we may assume that
both sets are convex. We set
$\mathcal{K}_\delta = \mathcal{K}^\delta_\mu\times\mathcal{K}^\delta_X$, we obtain the following estimate:
\begin{align}
 P\big[&\sup_{(u,\eta)\in \Uc\times I}\,\big\| h_\mu^{(N,J)}(u,\eta) - h_E^{(N)}(u,\eta)\big\|\geq \varepsilon/2 \big]\notag\\
 & \leq  P\big[ 1_{\mathcal{K}_{\delta}}\sup_{(u,\eta)\in \Uc\times I}\, \big\| h_\mu^{(N,J)}(u,\eta) - h_E^{(N)}(u,\eta)\big\|\geq \varepsilon/2 \big ] + 2\delta \notag\\
& \leq \sum_{j\in J}P\big[ 1_{\mathcal{K}_{\delta}}\sup_{(u,\eta_j)\in \Uc\times I_j}\, \Big\| \frac{1}{N}\sum_{i=1}^{N}\int f_j(u,\eta,X_i+ z)- f_j(u,\eta,X_i)\, d\mu_N(z)\Big\|\geq \varepsilon/2 \big ] + 2\delta. \label{e:consist1-smooth}
\end{align}
Since every continuous function is uniformly continuous on compact sets,
for each of the norms in \eqref{e:consist1-smooth} with $j=1,\dots,k, j\in J$, we have
$\big\|f_j(u,\eta,X_i+ z)- f_j(u,\eta,X_i)\big\| \leq w_{\delta}^j (\|z\|) .$
Taking  $w_\delta(t) = \max_{j\in J} w_{\delta}^j (t)$, we obtain a common modulus of continuity.
The following estimate holds for each term in the sum in \eqref{e:consist1-smooth}:
\begin{multline*}
 \Big\| \frac{1}{N}\sum_{i=1}^{N}\int_{\mathcal{K}_{\delta}} f_j(u,\eta,X_i+ z)- f_j(u,\eta,X_i)\, d\mu_N(z)\Big\|
 \leq\\
 \quad \frac{1}{N}\sum_{i=1}^{N}\int_{\mathcal{K}_{\delta}}\big\|f_j(u,\eta,X_i+ z)- f_j(u,\eta,X_i)\big\| d\mu_N(z)
         \leq  \int_{\mathcal{K}_{\delta}} w_{\delta}(\|z\|)\, d\mu_N(z).
\end{multline*}
Similar estimate is valid for the term involving $f_{k+1}$. Thus, we obtain
\begin{align}
\label{e:modulus_delta}
 P\big[&\sup_{(u,\eta)\in \Uc\times I}\,\big\| h_\mu^{(N,J)}(u,\eta) - h_E^{(N)}(u,\eta)\big\|\geq \varepsilon/2 \big]\leq \sum_{j\in J} P\big[ \int_{\mathcal{K}_{\delta}} w_{\delta}(\|z\|)\, d\mu_N(z)\geq \varepsilon/2 \big ] + 2\delta.
\end{align}
Notice that $w_{\delta}(\|\cdot\|)$ can be chosen continuous on $\mathcal{K}_{\delta}$ by the arguments in the proof of Theorem~\ref{metric-Lipschitz}. Therefore, $\int_{\mathcal{K}_{\delta}} w_{\delta}(\|z\|)\, d\mu_N(z)$
converges to zero, whenever $N\to \infty$ by the weak convergence of $\mu_N(z)$ and $w_{\delta}(0)=0$.
Thus, an $N_\delta$ exists such that for $N\geq N_\delta $, $\int_{\mathcal{K}_{\delta}} w_{\delta}(\|z\|)\, d\mu_N(z)< \varepsilon/2$. This together with \eqref{e:modulus_delta} shows that for $N\geq N_\delta $, we have
\begin{align*}
 P\big[\sup_{(u,\eta)\in \Uc\times I}\, \big\| h_\mu^{(N,J)}(\eta) - h_E^{(N)}(\eta)\big\|\geq \varepsilon/2 \big] &\leq  2\delta.
\end{align*}
Letting $\delta\downarrow 0$, we infer the uniform convergence in probability of $h_\mu^{(N,J)}(\cdot)$ to $\bar{f}(\cdot).$  Since convergence in probability is preserved by continuous mappings, we obtain the convergence in probability of the composition. The remaining part of the proof follows the same line of arguments as the proof of Theorem~\ref{metric-main}  using, additionally, the preservation of convergence in probability under continuous mappings.
  \end{proof}

Our results relate to the work on random approximations presented in \cite{vogel2005,vogel2017}, where the optimization problems do not involve composition but the feasible set may also be random and may be approximated based on sampled data.

\subsection{Wavelet-based estimators}
\label{s:wave}

We consider one more estimator based on alternative generalized kernel construction that was originally proposed in \cite{DP} for the case $m=1$ and was extended for arbitrary dimension in \cite{CSvS}.
Assuming that the distribution $P$ has a density, we define the following wavelet-based estimator for the density of $P$:
 \begin{equation}
 \label{wave}
\widetilde{d}_{N,j}(x)=\sum_{\ell\in\Zb} 2^{-j/2}\frac{1}{N}\sum_{i=1}^{N}\phi(2^{j}X_i-\ell)\phi_{j\ell}(x).
\end{equation}
Here the notation $\Zb$ stands for the set of integer numbers.

The function $\phi_{j\ell}$ is defined as  $\phi_{j\ell}(x)=2^{j/2}\phi(2^{j}x-\ell),
 $ where  $\phi (x)$ is
right-continuous, non-negative, with finite variation, and with a
compact support in an interval $[-a,a]$ with $ 1/2\le a<\infty .$ Furthermore,  the
following conditions are assumed:
\begin{itemize}
\item[(w1)]$\sum_{\ell\in\Zb} \phi(x-\ell) = 1$ for all $x \in \Rb$
\item[(w2)]
$x-\sum_{\ell\in\Zb} \ell \phi(x-\ell) = 0$ for all $x\in \Rb.$
\end{itemize}
Note that the condition (w1) also implies $\int_{\Rb } \phi(x)\,dx = 1$ (c.f. \cite{DPA}). The additional condition (w2) implies that $\int_{\Rb} x\phi(x)\,dx=0$ holds as well.
The integer $j$ is the resolution level and has to be chosen
appropriately in order to avoid over- or under-smoothing.

The suggested estimator \eqref{wave} is the empirical version
of the following approximation operator:
\begin{equation}
\label{appro}
T_j(h)=\sum_{\ell\in\Zb}\langle h,\phi_{j\ell}\rangle\phi_{j\ell}
\end{equation}
whereby each of the scalar products
$\langle h,\phi_{j\ell}\rangle = \int_{\Rb} h(t) \phi_{j\ell}(t)\,dt$
have been estimated using the data.
 The simplest possible locally linear choice of  $\phi(x)$  is
\begin{equation}
\label{linear}
\phi(x)=\begin{cases} 1+x &\text{ for } -1\le x < 0,\\
1-x &\text{ for }  0\le x\le 1,\\
0 & \hbox{ otherwise.}\\
\end{cases}
\end{equation}
 A smoother, locally quadratic  version of $\phi (x)$ is:
\begin{equation}
\label{quadratic}
\phi(x)=\begin{cases}0 &\text{ if  } |x| \ge 3/2,\\
0.5(1.5+x)^2, &\text{ if  }-3/2<x\le -1/2,\\
1+x-(x+0.5)^2, &\text{ if } -0.5<x<0.5,\\
0.5(1.5-x)^2, &\text{ if }1/2\le x\le 3/2.
\end{cases}
\end{equation}

In fact, a whole family of such estimators could be suggested by varying the choice of $\phi .$
They all have the advantage of being \emph{shape-preserving} estimators of the density, meaning that the resulting estimator is  non-negative and integrates to one.
Other asymptotically equivalent modifications of
$\widetilde{d}_{N,j}(x)$ are available that could serve our purpose equally well
but, admittedly, these modifications lack the simplicity of
\eqref{wave}.
It is important to notice that  the wavelet expansion
\eqref{wave} is \emph{non-orthogonal} in general. The shape-preservation property precludes
an orthogonal wavelet expansion being a continuous
function as discussed in \cite{DPA} and in \cite{DP}.

Several recommendations about the choice of the resolution level $j$    exist depending on the
assumptions on the cumulative distribution function of $X$ and the reader can find a detailed
discussion about those in \cite{DP} or \cite{CSvS}. For a large class of densities
the choice  $j \propto  \log_{2} N /5 $ ensures consistency.

Substituting $\tilde{d}_{N,j}$,  we get the estimator
\begin{equation}
\label{newestimatorwave}
\vartheta_w^{(N)} = \min_{u\in \mathcal{U}} \; \int_{\Xc} f(u,x)\tilde{d}_{N,j}(x)\,dx = \min_{u\in \mathcal{U}} \; \frac{2^j}{N} \sum_{i=1}^N
\int_{\Xc} f(u,x) K(2^jX_i,2^jx)\, dx,
\end{equation}
with the generalized kernel $K(y,x)=\sum_{\ell\in \mathbb{Z}}\phi (y-\ell)\phi(x-\ell).$

The constructions described above are based on a function $\phi $ of one-dimensional argument. They generalize algorithmically without  difficulty
to the multivariate case as demonstrated in (\cite{CSvS}). This can be done by using tensor-product wavelets. As a demonstration, let us take the case of dimension $m=2.$ Then we have data vectors $X_i=(X_{i_{1}},X_{i{2}}), i=1,\dots, N.$ For double indices $j=(j_1,j_2), \ell=(\ell_1,\ell_2)$  we define
\[
\left\{ \phi_{j\ell}(x)\right\}=\left\{ \phi_{j_1\ell_1}(x_1)\right\}_{\ell_1\in \mathbb{Z}}\left\{ \phi_{j_2\ell_2}(x_2)\right\}_{\ell_2\in \mathbb{Z}}.
\]
Then for a fixed resolution level vector $j=(j_1,j_2)$ we have the density approximation at the mixed resolution level $j=(j_1,j_2)$ given as
\[
\tilde{d}_{N,j}(x)=\sum_{\ell_1\in \mathbb{Z}}\sum_{\ell_2\in \mathbb{Z}}\Big[ \frac{1}{N}\sum_{i=1}^N \phi(2^{j_1}X_{i_1}-\ell_1)\phi(2^{j_2}X_{i_2}-\ell_2)\Big]\phi(2^{j_1}x_1-\ell_1)\phi(2^{j_2}x_2-\ell_2).
\]
Very often we would choose $j_1=j_2$ but this is not compulsory.

One should mention though that despite the easiness of the generalization for higher dimensions, the well-publicized curse of dimension is inherent to the estimation and the rates of convergence  of the optimal bandwidth are significantly lower. For example, the optimal rate in the estimation of the density is $O(N^{-1/(4+m)})$ and it deteriorates quickly with the dimension $m.$

The estimators obtained using \eqref{newestimatorwave} have similar properties to the usual kernel-based estimators. To spare space, we will not discuss these in this paper.
We can show that for a  variety of densities the  optimal compromise for the choice of $j$ is by achieved when equating the bias and variance terms which gives optimal order of $j^*\propto \log _2 N , $ with $j^*= \log _2 N /5 $ giving very good performance over a  wide variety of density classes.
Results that parallel our statements in Section \ref{s:empiricalestimators} can be formulated by using the uniform limit theorems for linear wavelet density estimators that have been discussed in detail in \cite{GinNickWave2009}. Although our translates  $\{\phi(\cdot -\ell )\}$ do not form an orthogonal system, the results of the cited paper cover the non-orthogonal case, too (see Remark 7 in particular).
We omit consistency results here, however, we shall show in the next section that our wavelet-based risk estimators improve the negative bias of the empirical estimator.

\section{Bias in composite optimization problems}
\label{s:bias}

As already mentioned,  SAA suffers from downward bias.
We can observe that problem \eqref{p:approximate} with $Q_N^J=P_N$ for certain compositions exhibits the same property, i.e., it underestimates the optimal value of problem \eqref{p:composite}.
Consider the problem of the form
\begin{gather*}
        \vartheta = \min_{u \in U} f_1(u,\Eb[f_2(u,X)])\\
        \vartheta_E^{(N)}=\min_{u\in U} f_1\left(u, \frac{1}{N}\sum_{i=1}^N f_2(u,X_i) \right)
    \end{gather*}
    where $U \subset \mathbb{R}^{n}$ is a nonempty compact set.

\begin{lemma}
\label{l:composite-bias-theta}
If the function $f_1(u,\cdot)$ is concave for all $u\in U$, then $\Eb[\vartheta_E^{(N)}]\leq \vartheta$.
\end{lemma}
\begin{proof}
Let a $X_1, X_2,\dots, X_N$ be a sample from the distribution of the random vector $X.$ Using the concavity of $f_1$ with respect to the second argument, we apply Jensen's inequality to obtain for all $u\in U$
\[
f_1(u,\Eb[f_2^j(u,X)]) = f_1\Big(u,\Eb\big[ \frac{1}{N}\sum_{i=1}^N f_2(u,X_i)\big]\Big)\geq
\mathbb{E} \Big[f_1\Big(u, \frac{1}{N}\sum_{i=1}^N f_2(u,X_i)\Big)\Big].
\]
Now, we use the fact that the minima of the functions follow the same relation. We obtain the following chain of inequalities:
\begin{multline*}
\vartheta = \min_{u\in U} f_1(u,\Eb[f_2(u,X)]) \geq
\min_{u\in U} \mathbb{E} \Big[ f_1\Big(u, \frac{1}{N}\sum_{i=1}^N f_2(u,X_i)\Big)\Big]\\
\geq \mathbb{E} \min_{u\in U} f_1\Big(u, \frac{1}{N}\sum_{i=1}^N f_2(u,X_i)\Big)=\mathbb{E} [\vartheta_E^{(N)}]. \qquad  
\end{multline*}
\end{proof}

Let us consider the portfolio optimization problem (\ref{port-msd}) by using inverse measures of risk.  The function $f_1(u,z,\cdot)$  there is concave and Lemma~\ref{l:composite-bias-theta} applies. This means that solving the portfolio optimization problem by empirical estimation will provide a biased estimation of the risk measure for the optimal portfolio by underestimating its riskiness.  We shall show that the (partially) smoothed estimators improve the bias under certain conditions.

In what follows, we assume that a kernel $K$ satisfying (k1)-(k2) is used for smoothing.
\begin{theorem}
\label{t:kernel-bias-comp1}
Given an index set $J\subset{1,\dots k+1},$ let $l=\max\{ j: j\in J\}$.
Assume that for all $j\in J$, the functions $f_j$ are convex with respect to the last argument and the functions $f_j(u,\cdot, x)$, $j=1,\dots \ell-1$ are monotonically non-decreasing.
The following inequality is satisfied:
\begin{equation}
\label{composite-bias-lower}
\begin{gathered}
\varrho_{E}^{(N)}[u,X]\leq {\varrho}_{K}^{(N,J)} [u,X]  \;\; \text{a.s.} \text{ for all } u\in U,\\
\vartheta_{E}^{(N)}\leq \vartheta_{K}^{(N,J)}\quad \text{a.s.}
\end{gathered}
\end{equation}
\end{theorem}

\begin{proof}
Consider $j\in J.$
We set $\frac{x-X_i}{h_N}=z$, $i=1, \cdots,N$ and use (k1) and the Jensen's inequality to obtain for all $i$ the following inequality:
\begin{equation*}
 \int f_{j}(u,\eta_j, x)K(\frac{x-X_i}{h_N})\frac{1}{h_N^m}dx
        = \int f_{j}(u,\eta_j,X_i+h_Nz)K(z)dz
        \geq f_{j}(u,\eta_j,X_i).
\end{equation*}
Taking the average, we get
\begin{equation*}
        \frac{1}{N}\sum_{i=1}^{N} \int f_{j}(u,\eta_j,x)K(\frac{x-X_i}{h_N})\frac{1}{h_N}dx
        \geq \frac{1}{N}\sum_{i=1}^N f_{j}(u,\eta_j,X_i).
\end{equation*}
If $j=k+1,$ then the last argument is omitted.
If $j>1$, then the monotonicity of the function $f_{j-1}(u,\cdot, x)$ entails the following a.s. inequality for all $\eta_j\in I_j$:
\begin{equation}
\label{j-ineq}
    f_{j-1}\left(u,\sum_{i=1}^N\frac{1}{N}\int f_j(u,\eta_j, x)K(\frac{x-X_i}{h_N})\frac{1}{h_N}dx, X \right)  \geq f_{j-1}\left(u, \frac{1}{N}\sum_{i=1}^N f_j(u,\eta_j, X_i), X \right).
\end{equation}
We state inequality \eqref{j-ineq} for $l-1$ and
repeat the argument for $j\in J$, $j<l$ using the monotonicity assumption. We obtain $ \varrho_{K}^{(N,J)}[u,X]\geq \varrho_{E}^{(N)} [u,X],$
    which is the first inequality in the statement.
We also observe that the minimum values of the two functions at the right and left hand side of the last chain of inequalities, taken with respect to $u\in U$, are related in the same way. Hence, we infer
  \[
  \min_{u \in U} \, \varrho_{K}^{(N,J)} [u,X] \geq \min_{u \in U}\,  \varrho_{E}^{(N)} [u,X]
    \]
Therefore,
    $\vartheta^{(N,J)}_K \geq \vartheta_E^{(N)}$ a.s. as stated.
  \end{proof}

Under the assumptions of the theorem, we have
\begin{equation}
\label{composite-bias-lowerbound}
\Eb[\vartheta_E^{(N)}]-\vartheta\leq \Eb[\vartheta_{K}^{(N,J)}] -\vartheta.
\end{equation}
Therefore a negative bias of the empirical estimator can be improved.\\

We denote by $\rho_w^{(N,J)}[u,X]$ the risk of the wavelet-based estimator of the risk and by $\vartheta_w^{(N)}$ the wavelet-based estimator of the optimal value $\vartheta.$
\begin{theorem}
\label{t:wavelet-bias-comp1}
Suppose the asumptions of Theorem~\ref{t:kernel-bias-comp1} are satisfied.
Also assume that the function $\phi(\cdot)$ in the construction of the estimator $\tilde{d}_{N,j^{*}}$ satisfies conditions (w1) and (w2) from Section \ref{s:wave}.
Then the following relations hold:
\begin{equation}
\label{wave-bias-lower}
\begin{gathered}
\varrho_{E}^{(N)}[u,X]\leq {\varrho}_{w}^{(N,J)} [u,X]  \;\; \text{a.s.} \text{ for all } u\in U,\\
\vartheta_{E}^{(N)}\leq \vartheta_{w}^{(N,J)}\quad \text{a.s.}
\end{gathered}
\end{equation}
\end{theorem}

\begin{proof}
Consider $j\in J.$
Fix $u\in U.$ For each $i=1,\dots, N$ in the representation
\[
\int_{\Xc} f_j(u,\eta_j,x)\tilde{d}_{N,j^{*}}(x)dx =  \frac{2^{j^{*}}}{N} \sum_{i=1}^N
\int_{\Xc} f_j(u,\eta_j,x) K(2^{j^{*}}X_i,2^{j^{*}}x) dx ,
\]
we apply the subgradient inequality at $X_i$. Let $g_j^i$ be any subgradient of $f_j(u,\eta_j,\cdot)$ at $X_i.$ We have $f_j(u,\eta_j,x)\geq f(u,\eta_j, X_i) + g_j^i(x-X_i)$ and we obtain
\begin{multline}
\frac{1}{N} \sum_{i=1}^N \int_{\Xc} f_j(u,\eta_j,x) K(2^{j^{*}}X_i,2^{j^{*}}x) dx
\geq \\
\frac{1}{N}\sum_{i=1}^N
\int_{\Xc} f_j(u,\eta_j,X_i) K(2^{j^{*}}X_i,2^{j^{*}}x) dx +
\frac{1}{N}\sum_{i=1}^N
\int_{\Xc} g_j^i(x-X_i) K(2^{j^{*}}X_i,2^{j^{*}}x) dx \\
=\frac{1}{N}\sum_{i=1}^N f_j(u,\eta_j,X_i)\sum_{\ell\in \mathbb{Z}}\phi(2^{j^{*}}X_i-\ell) +\qquad\\
\frac{1}{N}\sum_{i=1}^N \int_{\Xc} g_j^i\Big[\sum_{\ell\in \mathbb{Z}}2^{j^{*}}(x-X_i)\phi(2^{j^{*}}x-\ell)\phi(2^{j^{*}}X_i-\ell)\Big]dx.
\end{multline}
Under the Assumptions (w1)-(w2), we have obviously
$\sum_{\ell\in \mathbb{Z}}\phi(2^{j^{*}}X_i-\ell)=1.$
Hence the right-hand side simplifies further, i.e.,
\begin{multline}
\frac{1}{N}\sum_{i=1}^N \int_{\Xc} g_j^i \Big[\sum_{\ell\in \mathbb{Z}}2^{j^{*}}(x-X_i)\phi(2^{j^{*}}x-\ell)\phi(2^{j^{*}}X_i-\ell)\Big]dx =\\
\frac{1}{N}\sum_{i=1}^N  g_j^i \Big[\sum_{\ell \in \mathbb{Z}}  2^{-j^{*}}\phi(2^{j^{*}}X_i-\ell)\int_{\Xc}\phi(2^{j^{*}}x-\ell)(2^{j^{*}}x\pm \ell)d(2^{j^{*}}x) - \\
X_i \sum_{\ell\in\mathbb{Z}}\int\phi(2^{j^{*}} x-\ell)\phi(2^{j^{*}}X_i-\ell)\Big]d(2^{j^{*}}x-\ell)= \\
\frac{1}{N}\sum_{i=1}^N  g_j^i\Big[\sum_{\ell\in \mathbb{Z}}2^{-j^{*}}\phi(2^{j^{*}}X_i-\ell)\ell-X_i\Big] =0.
\end{multline}
Therefore,  we get that for each fixed $u$,
\[
\int_{\Xc}  f_j(u,\eta_j,x)\tilde{d}_{N, j^{*}}(x)dx \geq \frac{1}{N}\sum_{i=1}^N f_j(u,\eta_j,x).
\]
The remaining part of the proof follows precisely the same steps as in the proof of Theorem~\ref{t:kernel-bias-comp1}.
If $j=k+1,$ then the last argument is omitted.
If $j>1$, then the monotonicity of the function $f_{j-1}(u,\cdot, x)$ entails the following a.s. inequality for all $\eta_j\in I_j$:
\begin{equation}
\label{j-ineq-w}
    f_{j-1}\left(u,\tilde{d}_{N,{j^{*}}}, X \right)  \geq f_{j-1}\left(u, \frac{1}{N}\sum_{i=1}^N f_j(u,\eta_j, X_i), X \right).
\end{equation}
We state inequality \eqref{j-ineq-w} for $l-1$ and
repeat the argument for $j\in J$, $j<l$ using the monotonicity assumption. We obtain $ \varrho_{w}^{(N,J)}[u,X]\geq \varrho_{E}^{(N)} [u,X],$
    which is the first inequality in the statement.
We also observe that the minimum values of the two functions at the right and left hand side of the last chain of inequalities, taken with respect to $u\in U$, are related in the same way. Hence, we infer
  \[
  \vartheta^{(N,J)}_w= \min_{u \in U} \, \varrho_{w}^{(N,J)} [u,X] \geq \min_{u \in U}\,  \varrho_{E}^{(N)} [u,X] = \vartheta_E^{(N)}
    \]
    as stated.
  \end{proof}
Under the assumptions of the theorem, we have
$\Eb[\vartheta_E^{(N)}]-\vartheta\leq \Eb[\vartheta_{w}^{(N,J)}] -\vartheta,$ which implies again
that a negative bias of the empirical estimator can be improved by using the wavelet-like estimator.
\begin{theorem}
\label{t:kernel-bias-comp2}
Let a sample $X_1,X_2, \dots, X_N$ and an index set $J\subset{1,\dots k+1}$ be given. Suppose  that for each $j\in J$,
the function $f_j$ has a modulus of continuity $w_j^x(\cdot)$ with respect to the last argument and the function $f_{j-1}(u, \cdot, x)$, $j>1$, has a modulus of continuity $w_j^\eta(\cdot)$. Furthermore, assume that $w_j^x(\|\cdot\|)$ and $w_j^\eta(\|\cdot\|)$  do not dependent of the other arguments of the respective functions, and that they are integrable with respect to $K$. Then for every positive $\varepsilon$, a number $h_N^*$ exists such that for all $h\leq h_N^*$, it holds,
\begin{equation}
\label{bias-upperbound-epsilon}
\begin{gathered}
\varrho_{K}^{(N,J)}[u,X]\leq {\varrho}_E^{(N)} [u,X] + \varepsilon \;\; \text{a.s.} \text{ for all } u\in U,\\
{\vartheta}_{K}^{(N,J)}\leq {\vartheta}_E^{(N)} +  \varepsilon \;\; \text{a.s.}.
\end{gathered}
\end{equation}
If the moduli of continuity are of form $w^\eta_j(t)= \ell_j t^{\alpha_j}$ and
$w^x_j(t)= \ell_j^x t^{\beta_j}$ with $\beta_j$ less or equal to the order of the kernel $K$, then positive constants $L$ and $\alpha$ exist such that the following upper bound holds:
\begin{equation}
\label{bias-upperbound}
\begin{gathered}
\varrho_{K}^{(N,J)}[u,X]\leq {\varrho}_E^{(N)} [u,X] +  L h_N^\alpha \;\; \text{a.s.} \text{ for all } u\in U,\\
{\vartheta}_{K}^{(N,J)}\leq {\vartheta}_E^{(N)} +  L h_N^\alpha \;\; \text{a.s.}.
\end{gathered}
\end{equation}
\end{theorem}
\begin{proof}
Let $l$ be the largest index in $J$.
If $l=1$, then
\begin{multline*}
    \sum_{i=1}^N\frac{1}{N}\int f_1\left(u,\eta_1, X_i +h_Nz_1 \right) K(z_1)dz_1 - \; \sum_{i=1}^N\frac{1}{N} f_1\left(u, \eta_1, X_i \right)  \\
    = \sum_{i=1}^N\frac{1}{N}\int \Big[ f_1 \Big(u,\eta, X_{i_1} +h_Nz_1 \Big)   - f_1\Big(u, \eta, X_{i_1} \Big)
    \Big] K(z_1)dz_1
    \leq \int w_1^x(h_N\|z_1 \|) K(z_1)dz_1.
    \end{multline*}
Choosing the bandwidth small enough, we obtain the right-hand side of the  inequality smaller than any fixed number $\varepsilon >0.$\\
We analyse the case of $l>1$ and $l-1\not\in J$.
\begin{multline*}
    \sum_{i_1=1}^N\frac{1}{N}f_{l-1}\left(u,\sum_{i_2=1}^N\frac{1}{N}\int f_{l}(u,X_{i_2}+h_Nz_2)K(z_2)dz_2, X_{i_1}  \right) - \sum_{i_1=1}^N\frac{1}{N} f_{l-1}\left(u, \frac{1}{N}\sum_{i_2=1}^N f_{l}(u,X_{i_2}), X_{i_1} \right)  \\
    \leq  \sum_{i_1=1}^N\frac{1}{N}
  w_{l-1}^\eta\left(\Big\|\frac{1}{N}\sum_{i_2=1}^N\int f_{l}(u,X_{i_2}+h_Nz_2)K(z_2)dz_2 - \frac{1}{N}\sum_{i_2=1}^N f_{l}(u,X_{i_2}) \Big\| \right) \\
  \leq w_{l-1}^\eta\left(
\frac{1}{N}\sum_{i_2=1}^N\int w_l^x\big( \|h_Nz_2\|) K(z_2)dz_2\right)
   = w_{l-1}^\eta\left(
\int w_l^x\big(h_N\|z_2\|) K(z_2)dz_2\right).
\end{multline*}
Choosing small bandwidth $h_N$, we force $w_l^x\big(h_N\|z_2\|)\leq \varepsilon_1$ small enough to imply $w_{l-1}^\eta(\varepsilon_1)\leq \varepsilon$.\\
It remains to consider the case of $l>1$ and $l-1\in J$.
The following chain of inequalities holds:
\begin{align}
    \sum_{i_1=1}^N\frac{1}{N}&\int f_{l-1}\left(u,\sum_{i_2=1}^N\frac{1}{N}\int f_{l}(u,X_{i_2}+h_Nz_2)K(z_2)dz_2, X_{i_1} +h_Nz_1 \right) K(z_1)dz_1 \notag \\ \notag
    &\qquad\qquad\qquad  - \; \sum_{i_1=1}^N\frac{1}{N} f_{l-1}\left(u, \frac{1}{N}\sum_{i_2=1}^N f_{l}(u,X_{i_2}), X_{i_1} \right)  \\ \notag
    = &\sum_{i_1=1}^N\frac{1}{N}\int \Big[ f_{l-1} \Big(u,\sum_{i_2=1}^N\frac{1}{N}\int f_{l}(u,X_{i_2}+h_Nz_2)K(z_2)dz_2, X_{i_1} +h_Nz_1 \Big)  \\ \notag
    &\qquad\qquad\qquad -\; f_{l-1}\Big(u, \frac{1}{N}\sum_{i_2=1}^N f_{l}(u,X_{i_2}), X_{i_1} +h_Nz_1 \Big)
    \Big] K(z_1)dz_1 \\ \notag
    & + \sum_{i_1=1}^N\frac{1}{N}\int \Big[ f_{l-1} \Big(u,\sum_{i_2=1}^N\frac{1}{N} f_{l}(u,X_{i_2}), X_{i_1} +h_Nz_1 \Big)  - f_{l-1}\Big(u, \frac{1}{N}\sum_{i_2=1}^N f_{l}(u,X_{i_2}), X_{i_1} \Big)
    \Big] K(z_1)dz_1\\ \notag
    \leq & \sum_{i_1=1}^N\frac{1}{N}\int w_{l-1}^\eta\left(\Big\|\frac{1}{N}\sum_{i_2=1}^N\int f_{l}(u,X_{i_2}+h_Nz_2)K(z_2)dz_2 - \frac{1}{N}\sum_{i_2=1}^N f_{l}(u,X_{i_2}) \Big\| \right) K(z_1)dz_1\\ \notag
    & \qquad + \sum_{i_1=1}^N\frac{1}{N}\int w_{l-1}^x\left(\big\|h_Nz_1 \big\| \right) K(z_1)dz_1\\ \notag
  \leq & \int w_{l-1}^\eta\left(
\frac{1}{N}\sum_{i_2=1}^N\int w_l^x\big( \|h_Nz_2\|) K(z_2)dz_2\right) K(z_1)dz_1 + \int w_{l-1}^x(h_N\|z_1 \|) K(z_1)dz_1\\
   = & w_{l-1}^\eta\left(
\int w_l^x\big(h_N\|z_2\|) K(z_2)dz_2\right) + \int w_{l-1}^x(h_N\|z_1 \|) K(z_1)dz_1 \label{generalcase}
\end{align}
Again, choosing the bandwidth small enough, we obtain the right-hand side of the  inequality smaller than any fixed number $\varepsilon >0.$

Proceeding from $l$ to the smallest index in $J$ by the same line of arguments, we obtain the first part of the statements.

Suppose now  that $w_1^x(t) = \ell_1 t^{\alpha_1}$,  $w_2^x(t) =\ell_2 t^{\alpha_2}$
and $w_1^\eta(t) = \ell_3 t^{\alpha_3}.$
Without loss of generality, we may assume that $\ell_1=\ell_2=\ell_3=\ell.$
We obtain at the right-hand side of inequality \eqref{generalcase}
the following quantity:
\begin{align*}
w_1^\eta\Big( &
\int w_2^x\big(h_N\|z_2\|) K(z_2)dz_2\Big)+ \int w_1^x(h_N\|z_1 \|) K(z_1)dz_1 \\
&\; = w_1^\eta \left(\ell h_N^{\alpha_2}\int \|z_2\|^{\alpha_2} K(z_2)dz_2\right)  + \ell h_N^{\alpha_1}\int\|z_1 \|^{\alpha_1} K(z_1)dz_1 \\
&\; = \ell \left(\ell h_N^{\alpha_2}m_{\alpha_2}(K)\right)^{\alpha_3}
  + \ell h_N^{\alpha_1}m_{\alpha_1}(K)
\end{align*}
Assuming that $h_N<1$, we obtain
\[
w_1^\eta\left(
\int w_2^x\big(h_N\|z_2\|) K(z_2)dz_2\right)+ \int w_1^x(h_N\|z_1 \|) K(z_1)dz_1
\leq  \ell h_N^\alpha \big( \ell^{\alpha_3} m_{\alpha_2}^{\alpha_3}(K) + m_{\alpha_1}(K)  \big).
\]
Setting $L= \ell\big( \ell^{\alpha_3} m_{\alpha_2}^{\alpha_3}(K) + m_{\alpha_1}(K)  \big)$, we obtain inequality \eqref{bias-upperbound} for the composite functional in the case of two functions.
Consequently,
\begin{align*}
\min_{u\in U} \varrho^{(N,J)}_K[u,X] &\leq
\varrho^{(N,J)}_K[u_{E},X] \leq  \varrho^{(N)}_E[u_{E},X] + Lh_N^\alpha \\
& = \min_{u\in U} \varrho^{(N)}_E[u,X] + Lh_N^\alpha\quad\text{a.s.} &  
\end{align*}
\end{proof}

We have observed that when the modulus of continuity is of H\"older form, then the upper bound has more explicit dependence on the sample size.
Under the assumptions of Theorem \ref{t:kernel-bias-comp1} and Theorem \ref{t:kernel-bias-comp2}, the kernel estimator with sufficiently small bandwidth $h_N$ is less biased than the empirical one, when the latter has negative bias.
\begin{theorem}
Under the assumptions of Theorem \ref{t:kernel-bias-comp1} and Theorem \ref{t:kernel-bias-comp2} with H\"older modulus of continuity,
a positive number $h^*_N$ exists, such that whenever $h_N\in (0,h^*_N)$ the following relations holds with the constants $L$ and $\alpha$ from \eqref{bias-upperbound}.
\begin{align}
 &\;\big|\Eb[{\vartheta}_{K}^{(N,J)} ]-\vartheta\big|
\leq \big|\Eb[\vartheta_E^{(N)}]-\vartheta\big|.
\label{lessbias}\\
 &\; \Eb\big|\vartheta_{K}^{(N,J)} -\Eb[\vartheta_{K}^{(N,J)}]\big| \leq \Eb \big|\vartheta_E^{(N)}-\Eb[\vartheta_E^{(N)}] \big| +  Lh_N^\alpha,\label{dispertion}\\
 &\; \Big(\Eb[(\vartheta_{K}^{(N,J)} -\Eb[\vartheta_{K}^{(N,J)}] )^2]\Big)^{\frac{1}{2}} \leq \Big(\Eb (\vartheta_E^{(N)}-\Eb[\vartheta_E^{(N)} )^2\Big)^{\frac{1}{2}} +  Lh_N^\alpha\label{deviation}\\
&\;\Big(\Eb\Big[\big(\vartheta_K^{(N,J)} - \vartheta  \big)^2\Big]\Big)^{\frac{1}{2}}\leq \Big(\Eb\Big[\big(\vartheta_E^{(N)} - \vartheta  \big)^2\Big]\Big)^{\frac{1}{2}} + Lh_N^{\alpha}.
\label{MSE-upper}
\end{align}
\end{theorem}
\begin{proof}
If the empirical estimator ${\vartheta}_E^{(N)}$ has a bias $b<0$, then
 for any  $\varepsilon\leq -2b$, we have
 \[
 b=\Eb[{\vartheta}_E^{(N)}]-\vartheta\leq
  \Eb[{\vartheta}_{K}^{(N,J)} ]-\vartheta \leq
  \Eb[{\vartheta}_E^{(N)} ] -\vartheta - 2b  = -b
\]
Thus
$\big|\Eb[{\vartheta}_{K}^{(N,J)}]-\vartheta\big|
\leq \big|\Eb[{\vartheta}_E^{(N)}]-\vartheta\big|$,
showing \eqref{lessbias}.
Recall the relations \eqref{bias-upperbound} and \eqref{composite-bias-lower}:
\begin{equation}
\label{basic_pointwise}
{\vartheta}_E^{(N)}\leq \vartheta_K^{(N,J)} \leq  {\vartheta}_E^{(N)} + Lh_N^\alpha.
\end{equation}
This entails additionally,
$-\Eb[{\vartheta}_E^{(N)}]-Lh_N^\alpha \leq -\Eb[\vartheta_K^{(N,J)}] \geq  -\Eb[{\vartheta}_E^{(N)}].$
Adding the two relations together, we get
\begin{equation}
\label{mainpointwise}
{\vartheta}_E^{(N)}-\Eb[{\vartheta}_E^{(N)}]-Lh_N^\alpha \leq \vartheta_K^{(N,J)} -\Eb[\vartheta_K^{(N,J)}]\leq  {\vartheta}_E^{(N)}-\Eb[{\vartheta}_E^{(N)}] + Lh_N^\alpha.
\end{equation}
We consider $U=\vartheta_{K}^{(N,J)}-\Eb[\vartheta_{K}^{(N,J)}]$ and $V={\vartheta}_E^{(N)}-\Eb[{\vartheta}_E^{(N)}]$ and infer that
\begin{align*}
\Eb\big|\vartheta_K^{(N,J)} -\Eb[\vartheta_K^{(N,J)}]\big| & = \Eb|U| \leq
\Eb|V| +
\Eb|U-V| \leq \Eb|V| + Lh_N^\alpha\\
& =
\Eb\big|{\vartheta}_E^{(N)} -\Eb[{\vartheta}_E^{(N)}]\big| +Lh_N^\alpha,
\end{align*}
which proves inequality \eqref{dispertion}.
Further, we consider $U$ and $V$ as random variables in $\mathcal{L}_2$. Using \eqref{mainpointwise} again and the triangle inequality for the norm, we get
\begin{align*}
\Big(\Eb[(\vartheta_{K}^{(N,J)} -\Eb[\vartheta_{K}^{(N,J)}] )^2]\Big)^{\frac{1}{2}}
& =\|U\|_{\Lc_2} \leq \|V\|_{\Lc_2}+ \|U-V\|_{\Lc_2} \\
&\leq \|V\|_{\Lc_2}+ Lh_N^\alpha = \Big(\Eb[({\vartheta}_E^{(N)} -\Eb[{\vartheta}_E^{(N)}] )^2]\Big)^{\frac{1}{2}} + Lh_N^\alpha.
\end{align*}
This proves \eqref{deviation}.
Now, we turn to the upper bound on the error of the kernel estimator.
Using inequalities \eqref{basic_pointwise}, we get
\[
|\vartheta_{K}^{(N,J)} -\vartheta| \leq |{\vartheta}_E^{(N)}-\vartheta| +
|\vartheta_{K}^{(N,J)} -{\vartheta}_E^{(N)}| \leq |{\vartheta}_E^{(N)}-\vartheta| + Lh_N^\alpha.
\]
Denoting $Y=|\vartheta_{K}^{(N,J)} -\vartheta|$ and $W=|{\vartheta}_E^{(N)}-\vartheta|$, we view $Y$ and $W$ as two non-negative random variables in $\mathcal{L}_2$.
Recall that $\mathcal{L}_2$ is a Banach lattice and therefore, the relation between $Y$ and $W$ entails the same relation between their
$\Lc_2$-norms, i.e.,
\[
\Big(\Eb\big[(\vartheta_{K}^{(N,J)} -\vartheta])^2\big]\Big)^{\frac{1}{2}}
= \|Y\|_{\Lc_2}\leq \|W\|_{\Lc_2} + Lh_N^\alpha =
\Big(\Eb\big[({\vartheta}_E^{(N)} -\vartheta])^2\big]\Big)^{\frac{1}{2}} + Lh_N^\alpha.
\]
This shows \eqref{MSE-upper} and completes the proof.
  \end{proof}

We note that the same results apply for the relations between ${\vartheta}_E^{(N)}$
and $\vartheta_{K}^{(N,J)},$ as defined in the beginning of this section.
\bigskip

\section{Applications to measures of risk}
\label{s:application}

We return to the examples discussed at the beginning of section~\ref{s:empiricalestimators}.  We have considered a mean-semi-deviation of order $p\ge 1$ and a higher order inverse measure of risk for a portfolio problem. We shall verify the applicability of our consistency results.
Recall that the mean-semi-deviation portfolio optimization has the objective
\[
\varrho_1[u,X] = -\Eb[u^\top X] + \kappa \Big[\Eb\big[\big(\max\{0,\Eb[u^\top X]-u^\top X\}\big)^p\big]\Big]^{\frac{1}{p}},\vspace*{-1ex}
\]
where $\kappa \in [0,1]$ and the random vector $X$ comprises the random returns of the potential investments.
We have represented this measure as a composition of three functions:
\begin{align*}
f_1(u,\eta_1,x) = - u^\top x + \kappa\eta_1^{\frac{1}{p}},\quad
f_2(u,\eta_2,x) = \big(\max\{0, \eta_2 -u^\top x\}\big)^p,\quad
f_3(u,x) = u^\top x. 
\end{align*}
The empirical, the smoothed, and the kernel estimators are consistent because
the assumptions of Theorem ~\ref{t:consistency} are obviously satisfied.

In order to obtain strong consistency, or to verify the strong law of large numbers for the smoothed estimators, we need to discuss the uniform continuity assumption. The functions $f_1$, and $f_3$ have a modulus of continuity $w(t)= c_1t$, where $c_1$ is the maximum of the norm of $u\in U$. Indeed,
\[
|f_1(u,\eta_1,x+ z) - f_1(u,\eta_1,x)| = |u^\top z |\leq c_1\|z\|,
\]
where the last inequality is the Cauchy-Schwartz inequality. The modulus is independent of $u$ and $\eta_1$ and it is $K$-integrable for any kernel $K$, satisfying assumptions (k1) and (k2) with $s>1$.

The function $f_2(u,\eta_2,x)$ for $p>1$ is continuously differentiable with respect to $x$. We use the mean-value theorem with $\zeta\in [0,1]$ to obtain
\begin{align*}
\Big|\big(\max\{0, \eta_2 -\langle u, x+z\rangle\}\big)^p & - \big(\max\{0, \eta_2 -\langle u, x\rangle\}\big)^p\Big| = |\langle\frac{\partial}{\partial x} f_2(u,\eta_2,x+\zeta z), z\rangle| \\
&\leq p\max |\eta_2-\langle u, x+\zeta z\rangle|^{p-1}
|\langle u, z\rangle|\\
&\leq p\max \big( |\eta_2| + \|u\|\|x+\zeta z\|)^{p-1}\| u\|\| z\|\\
&\leq pc_1 \max \big(1,(|c_2 + c_1\|x\|+c_1\|z\|)^{p-1}\big)\| z\|.
\end{align*}
In the last inequality, we have used the compactness of $I_2$, denoting $c_2 =\max\{ |\eta_2|; \eta_2\in I_2\} $ The existence of the modulus of continuity is guaranteed  when the second argument is in a compact set.
We also see, that we need a kernel with finite $p$-moment in order to apply our results.
Therefore, it may be useful to smooth only the first and the third function instead of all three functions.

In the case of $p=1$, we have the same modulus of continuity for  $f_2$ as for $f_1$ and  $f_3$. Indeed,
\[
\Big|\max\{0, \eta_2 -\langle u, x+z\rangle\} - \max\{0, \eta_2 -\langle u, x\rangle\}\Big| |\langle u, z\rangle |\leq c\|z\|.
\]
Considering the inverse measures of risk
$\varrho[Y]=\min_{z \in \mathbb{R}}\left\{ z+\frac{1}{\alpha} [\mathbb{E}(\max(0,Y-z)^q)]^{1/q}\right\},$
we note that similar analysis can be carried out as for the mean-semi-deviation model. In this case, the assumptions of Theorem~\ref{t:consistency} are satisfied while the strong consistency
results are applicable for the bounded case.
We conclude that the consistency results hold for risk-averse portfolio optimization using those risk measures.

Additionally, consider the bias when the portfolio problems \eqref{port-msd} and \eqref{port-hor} are based on sampled data.
We see that the assumptions of Theorem \ref{t:kernel-bias-comp1} are satisfied with smoothing all or some functions in the composition.
The function $\eta\to \big(\max\{0, \eta -u^\top x\}\big)^p$ is uniformly continuous with respect to $\eta$ in a compact set. We analyze the dependence of the modulus of continuity on $x$ in a similar way as above. We have
\begin{align*}
\Big|\big(\max\{0, \eta_1 -\langle u, x\rangle\}\big)^p & - \big(\max\{0, \eta_2 -\langle u, x\rangle\}\big)^p\Big| \leq p\max |\eta-\langle u, x\rangle|^{p-1}|\eta|\\
&\leq p\max \big( |\eta| + \|u\|\|x\|)^{p-1}|\eta|\\
&\leq pc_1 \max \big(1,(|c_2 + c_1\|x\|)^{p-1}\big)| \eta |.
\end{align*}
The assumptions of Theorem \ref{t:kernel-bias-comp2} are also satisfied with Lipschitz modulus of continuity when $\Xc$ is bounded.


\section{Numerical Results}
\label{s:numerical}

In our numerical study, we have used the higher order measure of risk.
\begin{equation}
\label{horm}
\varrho[X]=\min_{u \in \mathbb{R}}\Big\{ u+\frac{1}{\alpha} [\mathbb{E}(\max(0,X-u)^q)]^{1/q}\Big\},
\end{equation}
with parameter  $q=2$ and $\alpha=0.05.$ The functions representing $\varrho$  are:
\[
f_1(u,y)=u+\frac{1}{\alpha}u^{1/q},\qquad
 f_2(u,x)=[\max(0,x-u)]^q.
\]
We recall the resulting  plug-in estimator
\[
{\vartheta}_E^{(N)} =\min_{u\in \mathbb{R}} \left\{ u+\frac{1}{\alpha}\left[ \frac{1}{N}\sum_{i=1}^N (\max (0,X_i-u))^q\right]^{1/q}\right\}.
\]
As shown in \cite{dentcheva2017statistical}, the empirical estimator is asymptotically normally distributed. However, as discussed in this paper, it exhibits a pronounced downward bias. We have compared ${\vartheta}_E^{(N)}$ to the following two estimators:
\begin{gather}
\vartheta_K^{(N)}=\min_{u\in \mathbb{R}} \left\{ u+\frac{1}{\alpha}\left[ \frac{1}{N}\sum_{i=1}^N \int (\max (0,x-u))^q\frac{1}{h}K(\frac{x-X_i}{h})dx\right]^{1/q}\right\}
\\
\vartheta_w^{(N)}=\min_{u\in \mathbb{R}} \left\{ u+\frac{1}{\alpha}\left[  \int (\max (0,x-u))^q\tilde{d}_{N,j}(x)dx\right]^{1/q}\right\}.
\end{gather}

All across in the simulations below, we have estimated the bias based on 500 simulations. For the wavelet estimator, we have experimented with both  (\ref{linear}) and (\ref{quadratic}) forms of $\phi .$ Both of them are first order splines  in the sense that they are only orthogonal to polynomials of degree one.
For the kernel estimator, we have experimented with the Gaussian kernel,
the uniform kernel $K(x)=\frac{1}{2h_N}$ with support on $|x| \leq h_N$, and
the Epanechnikov kernel $K(x)=\frac{3}{4}(1-x^2)$ on the support: $|x| \leq h_N$.


In the first series of experiments, we simulated observations from  the normal distribution $X\sim \mathcal{N}(10,3)$. We took $\alpha=0.05$ and $q=2. $ The optimal $u^*=14.5048$ is determined using numerical integration end the resulting ``true" $\vartheta_0=15.5163$ as the value of the risk.

We use the bandwidth calculated according to the formula $1.06 \hat{\sigma} N^{-\frac{1}{5}}$, where $\hat{\sigma}$ is the estimated standard deviation of the data.

The numerical results for the kernel estimator are reported in the following tables.

\bigskip

\begin{center}
\begin{tabular}{ |c | c| c | c |c |c  |}
 \hline
 N &  kernel  & bias-kernel & variance-kernel & bias-plug-in & variance-plug-in\\
 \hline
 100 &  Uniform & -0.6095  & 0.5893 & -1.1896 & 0.5754  \\
  \hline
200 &  Uniform & -0.3930 & 0.5132 & -0.7891 & 0.5350 \\
  \hline
 500 &  Uniform & -0.1655  & 0.3482 & -0.3236 &  0.4099  \\
  \hline
  100 &   Epanechnikov & -0.7254  & 0.5813 & -1.1896 & 0.5754  \\
  \hline
 200 &  Epanechnikov & -0.4852  & 0.5168 & -0.7891 & 0.5350  \\
  \hline
 500 &   Epanechnikov &  -0.2164  & 0.3641 & -0.3236 & 0.4099  \\
  \hline
 100 &  Gaussian & -0.6095  &  0.5893 & -1.1896 &  0.5754 \\
  \hline
 200 &  Gaussian & -0.3930  &  0.5132 & -0.7891 &  0.5350 \\
 \hline
 500 &  Gaussian & -0.1655  & 0.3482 & -0.3236 &  0.4099  \\
  \hline
\end{tabular}
\end{center}

\bigskip

In our numerical experiments, we tried to check if using the wavelet-based estimator with the universal  resolution level  $j^*=\log _2 N /5 $ offers a good solution. We have the following results:

\bigskip

\begin{center}
\begin{tabular}{ |c | c | c |c |c  |}
 \hline
 N &   bias-wavelet & variance-wavelet & bias-plug-in & variance-plug-in\\
 \hline
 100 &   -0.6430  & 0.6054 & -1.1668 & 0.6375  \\
  \hline
200 &   -0.3728 & 0.4879 & -0.7677 & 0.5382 \\
  \hline
 500 &   -0.1016  & 0.2842 & -0.2996 &  0.3525 \\
 \hline\end{tabular}
 \end{center}

\bigskip

These outcomes  also confirm that the results obtained by the kernel method and by the generalized kernel, the wavelet-based method, are very close.


In a second series of experiments, we used the $t$ distribution  with various degrees of freedom $\nu$ such as  6, 8 and 60,
with the data shifted to have the same mean of $10$ as the normal simulated data before. The performance with respect to bias reduction was
similar to the normal case. We present below the outcomes for the same parameters of the risk measure and some combinations of degrees of freedom and sample sizes. The variances of the $t$-distributed variables are finite and even smaller
than the variance of the normal random variables before. However, the heavier tails
 of the $t$ distribution in combination with the nonlinearity of the transformation defining the risk adversely affect the quality of the limiting normal approximation
  and as a consequence, a bias correction is welcome. It was pleasing to see that  the  same universal  resolution level  $j^*=\log _2 N /5 $ that was used in the normal case did a very good job also in the case of the $t$ distribution. As before, another pleasing aspect of the procedure was that the bias reduction was accompanied by a slight reduction of the variance, as well.

\bigskip

\begin{center}
\begin{tabular}{|c  c| c | c |c |c  |}
 \hline
 N &   df&bias-wavelet & variance-wavelet & bias-plug-in & variance-plug-in\\
  \hline
 100 &  6& -1.6239  & 1.3681 & -2.1477 & 1.4114  \\
  \hline
 200 &  6& -1.2090  & 1.4265 & -1.5892 & 1.4979  \\
  \hline
 500 &  6& -0.5870  & 1.9387 & -0.7453 & 2.1290  \\
  \hline
 100 &  8& -1.0266  & 0.9092 & -1.5532 & 0.9622  \\
  \hline
200 &   8&-0.6814 & 0.9434 & -1.0694 & 1.0175 \\
  \hline
 500 &  8& -0.3029  & 1.0214 & -0.480 &  1.1519 \\
 \hline
 100 &  60& -0.2176  & 0.2182 & -0.7692 & 0.2506  \\
  \hline
200 &   60&-0.0788 & 0.1745 & -0.5058 & 0.2171 \\
  \hline
 500 &  60& 0.0490  & 0.0935 & -0.2092 &  0.1366 \\
 \hline
 \end{tabular}
 \end{center}

\medskip

The results for the kernel estimator are reported in the following tables
\medskip

\begin{center}
\begin{tabular}{ |c  |c | c| c | c |c |c  |}
 \hline
 N & dg &kernel  & bias-kernel & variance-kernel & bias-plug-in & variance-plug-in\\
 \hline
 100 & 6& Uniform & -1.9800   & 1.3440 & -2.1343 & 1.3150   \\
  \hline
200 & 6& Uniform & -1.4528  & 1.5973 & -1.5649 & 1.5886  \\
  \hline
 500 & 6& Uniform & -0.7694   & 1.6350 & -0.7952 &  1.6624   \\
  \hline
  100 & 8&  Uniform & -1.4044   & 1.2057 & -1.5452 & 1.1805  \\
  \hline
 200 & 8&  Uniform & -0.9433   & 1.2299 & -1.0468 & 1.2207  \\
  \hline
 500 & 8&  Uniform &  -0.4875   & 1.0281 & -0.5126 & 1.0460  \\
  \hline
 100 & 60& Uniform & -0.6193   &  0.2529 & -0.7367 &  0.2457 \\
  \hline
 200 & 60& Uniform & -0.3776   &  0.2168 & -0.4642 &  0.2158 \\
 \hline
 500 & 60& Uniform & -0.1513   & 0.1687 & -0.1789 &   0.1768  \\
  \hline
\end{tabular}
\end{center}

\bigskip

\begin{center}
\begin{tabular}{ |c  |c | c| c | c |c |c  |}
 \hline
 N & dg &kernel  & bias-kernel & variance-kernel & bias-plug-in & variance-plug-in\\
 \hline
 100 & 6&  Epanechnikov  & -2.0119    & 1.3370 & -2.1343 & 1.3150   \\
  \hline
200 & 6&  Epanechnikov  & -1.4790   & 1.5954 & -1.5649 & 1.5886  \\
  \hline
 500 & 6&  Epanechnikov  & -0.7782    & 1.6435 & -0.7952 &  1.6624   \\
  \hline
  100 & 8&  Epanechnikov  & -1.4336    & 1.1996 & -1.5452 & 1.1805  \\
  \hline
 200 & 8&   Epanechnikov  & -0.9675    & 1.2299 & -1.0468 & 1.2207  \\
  \hline
 500 & 8&   Epanechnikov  &  -0.4960   & 1.0336 & -0.5126 & 1.0460  \\
  \hline
 100 & 60&  Epanechnikov  & -0.6436    &  0.2510 & -0.7367 &  0.2457 \\
  \hline
 200 & 60&  Epanechnikov  & -0.3979   &  0.2166 & -0.4642 &  0.2158 \\
 \hline
 500 & 60&  Epanechnikov  & -0.1606    & 0.1710 & -0.1789 &   0.1768  \\
  \hline
\end{tabular}
\end{center}

\bigskip




\section{Conclusion}

In a summary, we have suggested a simple and  computationally inexpensive procedure to reduce the slowly diminishing downward  bias that plagues the solutions to modern risk management problems where the risk functional as a rule is a composite one. This procedure is of significant practical interest.
Besides the theoretical backing of the risk reduction effect provided by our smoothed estimators, we have demonstrated the effect by comparing
empirical sample-based optimization of composite functionals and kernel-based composite optimization. We also observed numerically the pleasant effect that the bias reduction is accompanied by a slight reduction of the variance of the risk estimator.

Our theoretical results indicate that, if the goal is to reduce the bias of the empirical estimator ${\vartheta}_E^{(N)}$ by replacing it with a kernel-based $\vartheta_K^{(N)}$ then under mild conditions on the order of the optimal bandwidth for the kernel-based estimator of $\vartheta$  coincides with the order of the optimal bandwidth for the kernel estimator of the density of the data if we assumed that one exists. Similarly, if we choose to use the generalized kernel-based estimator of $\vartheta$  for the same purpose then the order of the optimal resolution level coincides with the order of the optimal resolution level for the generalized kernel estimator of the density of the data.

The development of a parallel theory about the optimal with respect to the chosen kernel bandwidth is beyond the scope of this paper. Including the composition of functions in the search would result in a problem whose level of difficulty is higher than identifying the optimal bandwidth.
We also point out to the discussion in \cite[Section 1.2.4]{tsybakov2008introduction}, which shows that the approach on determining optimal bandwidth for a fixed data density can be seriously criticized.
 For practical purposes, one can choose the bandwidth that is recommended for estimating the data density and plug it in the estimator $\vartheta_K^{(N)}$. Analogously, we can choose the optimal resolution level $j^*$ and plug it in the estimator $\vartheta_w^{(N)}.$

 Many kernels and wavelet-based generalized kernels could be used to this end, with a variety of results presented in Section 7. The wavelet-based procedure has some advantages in that the suggested choice of resolution level seems to be universally valid (and hence more robust) for a large classes of distributions of $X.$ The  use of the other types of kernels requires more careful tailoring of the choice of the kernel bandwidth depending on the (typically unknown) distribution of $X.$ In that sense using these kernels represents a less robust approach in comparison to the generalized wavelet kernel choice.

\end{document}